        \documentclass[11pt,letterpaper]{amsart}
\usepackage{amssymb, amsmath, amsthm, esint}
\usepackage{mathrsfs}
\usepackage{bbm,dsfont}
\usepackage{hyperref}
\usepackage{mathtools}
\usepackage{fullpage}
\usepackage{color}

\title[Discrete Stein-Wainger]
{Discrete Analogues in Harmonic Analysis: A Theorem of Stein-Wainger}

\author[B. Krause]{Ben Krause}
\address{BK: Department of Mathematics, King's College London, WC2R 2LS, UK}
\email{ben.krause@kcl.ac.uk}

\date{\today}


\theoremstyle{plain}
\newtheorem{mthm}{Theorem}
\newtheorem{mcor}[mthm]{Corollary}
\newtheorem{thm}{Theorem}[section]
\newtheorem{proposition}[thm]{Proposition}
\newtheorem{lemma}[thm]{Lemma}

\newtheorem{definition}[equation]{Definition}
\newtheorem{example}{Example}
\theoremstyle{remark}

\numberwithin{equation}{section}
\begin{document}
\maketitle
\begin{abstract}
For $d \geq 2, \ D \geq 1$, let $\mathscr{P}_{d,D}$ denote the set of all degree $d$ polynomials in $D$ dimensions with real coefficients without linear terms. We prove that for any Calder\'{o}n-Zygmund kernel, $K$, the maximally modulated and maximally truncated discrete singular integral operator,
\begin{align*}
 \sup_{P \in \mathscr{P}_{d,D}, \ N} \Big| \sum_{0 < |m| \leq N} f(x-m) K(m) e^{2\pi i P(m)}  \Big|,
\end{align*}
is bounded on $\ell^p(\mathbb{Z}^D)$, for each $1 < p < \infty$. Our proof introduces a stopping time based off of equidistribution theory of polynomial orbits to relate the analysis to its continuous analogue, introduced and studied by Stein-Wainger:
\begin{align*}
 \sup_{P \in \mathscr{P}_{d,D}} \Big| \int_{\mathbb{R}^D} f(x-t) K(t) e^{2\pi i P(t)} \ dt \Big|.
\end{align*}
\end{abstract}

 \setcounter{tocdepth}{1}
\tableofcontents 

\section{Introduction}

This paper will be concerned with a so-called discrete Carleson-type operator, namely the maximally modulated and maximally truncated discrete singular integral operator,
\begin{align}\label{e:disc0}
\sup_{P \in \mathscr{P}_{d,D}, \ N} \Big| \sum_{0 < |m| \leq N} f(x-m) K(m) e^{2\pi i P(m)} \Big|
\end{align}
where $K : \mathbb{R}^D \to \mathbb{C}$ is a normalized Calder\'{o}n-Zygmund kernel: $K \in \mathcal{C}^1(\mathbb{R}^D \smallsetminus 0)$ with $\| K \|_{\text{CZ}(\mathbb{R}^D)} \leq 1$, i.e.\
\begin{align}\label{e:cznorm}
\sup_{0 < r < R} |\int_{r\leq |x| \leq R} K(x) \ dx|  + \sup_{x \neq 0} \ |x|^D \cdot |K(x)| + \sup_{x \neq 0} \ |x|^{D+1} \cdot |\nabla K(x)| \leq 1
\end{align}
with the left side of \eqref{e:cznorm} defining the norm $\| K \|_{\text{CZ}(\mathbb{R}^D)}$. For $d \geq 2, \ D \geq 1$, 
\begin{align}\label{e:polysdD}
\mathscr{P}_{d,D} := \Big\{ \sum_{2 \leq |\alpha| \leq d} \lambda_\alpha x^{\alpha} \in \mathbb{R}[x_1,\dots,x_D] \Big\} 
\end{align}
denote the set of all degree $d$ polynomials in $D$ dimensions with real coefficients without linear terms; see \eqref{e:multiindex} below for multi-index notation.

\subsection{History}
The study of maximally polynomially modulated singular integrals has a long and rich history, most notably encompassing Carleson's celebrated theorem on convergence of Fourier series \cite{C}; recent work of Lie \cite{Lie} (also see Zorin-Kranich \cite{ZK}) have essentially concluded this line of research: for every $1 < p <\infty$, the operator
\begin{align}\label{e:fullcarl}
    \mathscr{C}_{d,D} f(x) := \sup_{P,R} |\int_{|t| > R} f(x-t) K(t) e^{2\pi i P(t)}  \ dt| 
\end{align}
satisfies
\begin{align}\label{e:fullcarl}
    \| \mathscr{C}_{d,D} f \|_{L^p(\mathbb{R}^D)} \lesssim_{d,D,p} \| f \|_{L^p(\mathbb{R}^D)}, 
\end{align}
where the supremum is taken over all polynomials in $D$ variables, of degree at most $d$, and all $0 < R <\infty$. 

The major challenge presented in developing this theory is that the operator \eqref{e:fullcarl} is invariant under modulation by any polynomial of degree $\leq d$, a point which propagates throughout the analysis, and limits the efficacy of the standard Calder\'{o}n-Zygmund/Littlewood-Paley approach, in which the zero frequency has a distinguished role in the analysis. Indeed, if $Q(t)$ is a polynomial of degree $\leq d$, then for any $P$ of degree $\leq d$
\begin{align*}
&\int_{|t| > R} \Big( f(x-t) e^{2 \pi i Q(x-t)} \Big) K(t) e^{2\pi i P(t)}  \ dt \\
&= e^{2 \pi i Q(x)} \cdot \int_{|t| > R} \Big( f(x-t) e^{2 \pi i \big( Q(x-t) - Q(x) \big)} \Big) K(t) e^{2\pi i P(t)}  \ dt \\
&\qquad  \qquad =: e^{2 \pi i Q(x) } \cdot \int_{|t| > R} f(x-t) K(t) e^{2\pi i P_x(t)} \ dt 
\end{align*}
where $t \mapsto P_x(t)$ has degree $\leq d$, uniformly in $x$. Consequently,
\begin{align*}
\Big| \int_{|t| > R} \Big( f(x-t) e^{2 \pi i Q(x-t)} \Big) K(t) e^{2\pi i P(t)}  \ dt \Big| \leq \mathscr{C}_{d,D} f(x);
\end{align*}
taking suprema, we see that
\begin{align*}
\mathscr{C}_{d,D} \big( f \cdot e^{2 \pi i Q( \cdot) } \big)(x) = \mathscr{C}_{d,D} f(x).
\end{align*}

Earlier, Stein and Wainger \cite{SW} investigated the analogue of \eqref{e:fullcarl} when this modulation invariance is eliminated, and the role of the zero frequency remains appropriately distinguished. In particular, they were interested in understanding the following operator:
\begin{align}\label{e:SW00}
    C_{d,D} f(x) := \sup_{P \in \mathscr{P}_{d,D}} \Big| \int  f(x-t) K(t) e^{2\pi i P(t)} \ dt \Big|
\end{align}
where, now, $\mathscr{P}_{d,D}$ denotes \eqref{e:polysdD}, the set of all degree $d$ polynomials in $D$ dimensions with real coefficients without linear terms, and $K$ is a normalized Calder\'{o}n-Zygmund kernel satisfying \eqref{e:cznorm}. They were able to establish full $L^p$ estimates, for $1 < p < \infty$, for $C_{d,D}$ without resorting to modulation-invariant tools, but rather by relying on oscillatory integral techniques. In recent years, significant attention has been devoted to exploring their arguments, and a number of papers studying oscillatory integrals without modulation invariance in a wide variety of contexts have been written, see for instance \cite{G+, GRY,KL,BK11,BK12,P, Ramos0, Ramos1}. Indeed, this work was explicitly discussed in \cite{F+}, a review of Stein's major mathematical contributions.

\subsection{Discrete Harmonic Analysis}
Recently, the study of maximally modulated singular integrals without modulation invariance has been conducted in the integer setting. This study was initiated in \cite{KL}, where the analogue of Stein's purely quadratic Carleson operator \cite{S} was introduced and studied:
\begin{align}\label{e:disc1}
    \sup_\lambda \Big| \sum_{m \neq 0} f(x-m) \frac{e^{2\pi i \lambda m^2}}{m} \Big|.
\end{align}
At this point, the theory of polynomial Radon transforms, initiated by Jean Bourgain \cite{B0,B2,B1}, had become well-developed, see \cite{MST1,MST2,MSW}; the operator \eqref{e:disc1} was the first example of a discrete analogue in harmonic analysis that was not of this form. Accordingly, it proved surprisingly resistant to early attempts to bound it, even on $\ell^2(\mathbb{Z})$, see \cite{KL}. These difficulties were later resolved in \cite{BK11}, where a full $\ell^2(\mathbb{Z})$ theory was developed, and more recently in \cite{BK12}, where \eqref{e:disc1} was shown to be bounded on  $\ell^p(\mathbb{Z})$ for each $1 < p < \infty$ as a special case of broader work concerning suprema over one parameter families of modulation parameters. 

\subsection{Main Results}
In this paper, we establish a discrete analogue of Stein-Wainger's result:

\begin{mthm}\label{t:main}
For every $1 < p < \infty$, the operator \eqref{e:disc0} is bounded on $\ell^p$:
\begin{align*}
\| \sup_{P \in \mathscr{P}_{d,D}, \ N} \Big| \sum_{0 < |m| \leq N} f(x-m) K(m) e^{2\pi i P(m)} \Big| \|_{\ell^p(\mathbb{Z}^D)} \lesssim_{d,D,p} \| f\|_{\ell^p(\mathbb{Z}^D)}.
\end{align*}
\end{mthm}

An immediate application of our theorem is to variable coefficient singular integrals. For 
\[ V = (v_1,\dots,v_k), \; \; \; v_1,\dots,v_k : \mathbb{Z}^D \to \mathbb{Z}^D,\]
 consider the variable coefficient singular integral operator
\begin{align*}
&T_V f(x_0,x_1,\dots,x_k) : \mathbb{Z}^{D\cdot (k+1)} \to \mathbb{C} \\
& \qquad := \sum_{m \neq 0, \ m \in \mathbb{Z}^D} f(x_0-m,x_1 - v_1(x_0) P_1(m),\dots,x_k - v_k(x_0) P_k(m) ) \cdot K(m),
\end{align*}
where $K$ is a normalized Calder\'{o}n-Zygmund kernel, see \eqref{e:cznorm}.

By taking a partial Fourier transform in the final $k$ variables, and applying Plancherel's theorem and the $p=2$ case of Theorem \ref{t:main}, we arrive at the following Corollary. 
\begin{mcor}
Suppose that $f \in \ell^2(\mathbb{Z}^{D(k+1)})$ and that $\{ P_i(m)  \} \subset \mathbb{Z}[m_1,\dots,m_D]$ are a collection of polynomials without linear terms and maximal degree $d$. Then for each $d,D,k$,
\begin{align*}
\sup_{V = (v_1,\dots,v_k)} \ \| T_V f \|_{\ell^2(\mathbb{Z}^{D(k+1)})} \lesssim_{d,D,k} \| f\|_{\ell^2(\mathbb{Z}^{D(k+1)})}.
\end{align*}
\end{mcor}

Theorem \ref{t:main}, along with the recent preprint \cite{B++} is the first multi-parameter result in discrete analogues in harmonic analysis: one can express
\begin{align*}
&\sup_{P \in \mathscr{P}_{d,D}, \ N} \Big| \sum_{0 < |m| \leq N} f(x-m) K(m) e^{2\pi i P(m)} \Big| \\
& \qquad = \sup_{\lambda_\alpha : \alpha \in \Gamma, \ N} \Big| \sum_{0 < |m| \leq N} f(x-m) K(m) e^{2\pi i \cdot \sum_\alpha \lambda_\alpha m^{\alpha} }\Big| 
\end{align*}
see \eqref{e:Gamma}; the multi-parameter nature of the supremum necessitated a different set of techniques than those used in \cite{KL,BK11,BK12}. The analysis here is much closer in spirit to that of Stein-Wainger, with the Kolmogorov-Seliverstov method of $TT^*$ playing a crucial role. The technical ingredient needed to address \eqref{e:SW00} was the oscillatory integral bound
\begin{align}\label{e:coeffnorm}
    |\int_{[0,1]^D} e^{2 \pi i P(t)} \ dt | \lesssim  (1 + \| P \|)^{-\theta}
\end{align}
and the corresponding sub-level estimate
\begin{align}\label{e:sublevel}
    |\{ t \in [0,1]^D : |P(t)| \leq \epsilon \}| \leq \big( \frac{\epsilon}{\|P\|} \big)^{\theta}
\end{align}
for real-valued polynomials
\begin{align*}
    P(t) = \sum_\alpha \lambda_\alpha t^\alpha 
\end{align*}
equipped with the coefficient norm
\begin{align}\label{e:Ecn}
    \| P \| := \sum_{|\alpha| \neq 0} | \lambda_\alpha|.
\end{align}
\eqref{e:sublevel} is often known as a \emph{non-concentration estimate}, as it says that it is very hard for the image of a polynomial to cluster disproportiantely near a single value. The optimal bound $\theta = 1/d$ is established in \cite{SW}, but the existence of any $\theta > 0$ would still be sufficient to establish their main result.

In this work, we develop an analogous mechanism for estimating exponential sums and use it to establish appropriate non-concentration estimates.

Classical equidistribution theory dictates that any polynomial $P: \mathbb{Z}^D \to \mathbb{R}$ of degree $\leq d$, and any (large) integer ${A}$, there exists some $M \lesssim_{{A} ,d} 1$ so that on the interval
\begin{align*}
\{ 1,\dots,N-1,N \}
\end{align*}
 one may decompose
\begin{align}\label{e:taodec}
P = P_{\text{Smooth}} + P_{\text{Equi}} + P_{\text{Rat}},
\end{align}
see \cite[Proposition 1.1.17]{T0}, so that -- essentially --
\begin{align}\label{e:stop}
|\sum_{n=1}^N e^{2 \pi i P(n)} | \lesssim M^{2} \max_{\frac{N}{M^{{A}}} \leq T \leq N} \Big| \sum_{n=1}^T e^{2 \pi i P_{\text{Equi}}(n)} \Big|  + M^{-A}
\end{align}
(say),
where $P_{\text{Equi}}$ is $M^{{A}}$-\emph{equidistributed}, in that
\begin{align*}
\sup_{\frac{N}{M^{A}} \leq T \leq N} \Big| \frac{1}{T} \sum_{n=1}^T e^{2 \pi i P_{\text{Equi}}(n)} \Big| = o_{M}(1)
\end{align*}
while $P_{\text{Smooth}}$ is smooth,
\begin{align*}
\| P_{\text{Smooth}} \|_{\text{Lip}} \leq \frac{M}{N}
\end{align*}
and $P_{\text{Rat}}$ has rational coefficients of bounded denominator, in that there exists $Q \leq M$ so that
\begin{align*}
Q \cdot P_{\text{Rat}} \equiv 0 \mod 1.
\end{align*}

The equation \eqref{e:stop} speaks to a negotation: the faster that $P$ equidistributes $\mod 1$, the smaller that the pertaining exponential sums become, but the longer it takes to achieve \eqref{e:taodec}, the less control over \eqref{e:stop} we enjoy. With this in mind, we introduce the scale-dependent stopping time
\begin{align}\label{e:coeffnorm00}
N_{N}(P) := \min\{ M : P_{\text{Equi}} \equiv 0 \text{ vanishes entirely} \}
\end{align}
and note that it is well-defined (a trivial upper bound is $N^{\text{deg}(P)}$). The significance of this quantity is that we may use inverse theorems to bound
\begin{align}\label{e:keyquant}
\Big| \frac{1}{N} \sum_{n=1}^N e^{2 \pi i P(n)} \Big| \lesssim N_N(P)^{-\theta}
\end{align}
for some $\theta = \theta_{A,d} > 0$, which acts as a substitute to \eqref{e:coeffnorm}; we use \eqref{e:coeffnorm00} to deduce appropriate arithmetic analogues of the estimate \eqref{e:sublevel}, which contains the key analytical input behind our approach.

The key quantitative estimate \eqref{e:keyquant} essentially appears as \cite[Lemma 1.16]{T0}. We will need a higher-dimensional version, which first appeared in \cite{Tblog}; we provide an alternative proof for completeness in our Appendix \ref{s:app}.

\subsection{Structure}
The structure of the paper is as follows:\\
\medskip

We begin by reviewing \cite{SW} in \S \ref{s:SW}, with a focus on the central role of the estimates \eqref{e:coeffnorm} and \eqref{e:sublevel}; 

In \S \ref{s:exp}, we introduce and discuss \eqref{e:coeffnorm00}, our analogue of the Euclidean coefficient norm. We use this technology to deduce some non-concentration estimates for polynomials;

In \S \ref{s:carl}, we reduce the study of \eqref{e:disc1} to the special case where the coefficients of $P$ live very close to cyclic subgroups with small denominators;

In \S \ref{s:nthry}, we use the circle method to approximate our operators, subject to the constraint on the coefficients of the polynomials. The arguments here are similar to those in \cite[\S 5]{BK11};

In \S \ref{s:finish}, we complete the proof.

Our Appendix \ref{s:app} contains the proofs of our exponential sums estimates.

\subsection{Notation}
Throughout, we let $e(t) := e^{2\pi i t}$ denote the complex exponential. For real $K > 0$, let
\begin{align}\label{e:mu}
    \mu_K(s) := \frac{1}{K} \cdot (1 - \frac{|s|}{K})_+
\end{align}
 be the one-dimensional Fej\'{e}r kernel at scale $K$, so that
\begin{align*}
    \widehat{\mu_K}(\xi) = \big( \frac{\sin(\pi K \xi)}{\pi K \xi} \big)^2.
\end{align*}

We use the induced norm on the Torus
\begin{align*}
\| x \|_{\mathbb{T}} := \min\{ |x-n| : n \in \mathbb{Z} \}.
\end{align*}

For multi-indices
\begin{align*}
    \alpha = (\alpha_1,\dots,\alpha_D), \; \; \; \alpha_i \in \mathbb{Z}_{\geq 0}
\end{align*}
we define
\begin{align}\label{e:multiindex}
    x^\alpha := \prod_{i=1}^D x_i^{\alpha_i}.
\end{align}
We let 
\begin{align}\label{e:coord}
 {e_j} := (0,\dots,0,1,0,\dots,0)
\end{align}
denote the coordinate vector with $1$ in the $j$th component.

We define the ordering on multi-indices $\beta \leq \alpha$ if $\beta_i \leq \alpha_i, \ 1 \leq i \leq D$. If at least one strict inequality holds, we use $\beta < \alpha$.

We let $|\alpha| := \sum_i \alpha_i$, and let 
\begin{align}\label{e:Gamma}
 \Gamma = \Gamma_{d,D} := \{ \alpha : 2 \leq |\alpha| \leq d\} 
\end{align}
so that
\[ P \in \mathscr{P}_{d,D}\]
precisely when $P$ is a linear combination of monomials with exponents in $\Gamma$.
Note the upper bound
\begin{align*}
    |\Gamma| \leq \binom{D+d}{D}.
\end{align*}

For $\lambda = \{ \lambda_\alpha : \alpha \in \Gamma \} \in \mathbb{R}^{|\Gamma|}$, we use the notation
\begin{align}\label{e:poly}
P_\lambda(x) := \sum_\alpha \lambda_\alpha x^\alpha.
\end{align}

We will let $[R] := [-R,R]$ and $(R) := (0,R]$, with context determining whether we restrict to integers or not.

We let
\begin{align*}
    \vec{R} = (R_1,\dots,R_D)
\end{align*}
and define
\begin{align*}
    \vec{R}^\alpha := \prod_{i=1}^D R_i^{\alpha_i},
\end{align*}
and
\begin{align*}
|\vec{R}| := \prod_{i=1}^D R_i.
\end{align*}
We also set
\begin{align}\label{e:box}
[\vec{R}] = [R_1] \times \dots \times [R_D].
\end{align}

We will make use of the modified Vinogradov notation. We use $X \lesssim Y$, or $Y \gtrsim X$ to denote the estimate $X \leq CY$ for an absolute constant $C$ and $X,Y \geq 0$. If we need $C$ to depend on a parameter,
we shall indicate this by subscripts, thus for instance $X \lesssim_p Y$ denotes
the estimate $X \leq C_p Y$ for some $C_p$ depending on $p$. We use $X \approx Y$ as
shorthand for $Y \lesssim X \lesssim Y $. We reserve the notation
\begin{align}\label{e:sim}
    X \sim Y :=  \Big( X \in [Y/2, Y) \Big)
\end{align}
to denote the inequality $Y/2 \leq X < Y.$

We also make use of big-O notation: we let $O(Y)$ denote a quantity that is $\lesssim Y$, and similarly $O_p(Y)$ a quantity that is $\lesssim_p Y$.

\section{A Reveiw of Stein-Wainger}\label{s:SW}
Recall the Stein-Wainger maximal operator,
\begin{align}\label{e:SW}
    C_{d,D} f(x) := \sup_{\mathscr{P}_{d,D}} \Big| \int  f(x-t) K(t) e(P(t)) \ dt \Big|.
\end{align}

In this section we review the argument of \cite{SW} to prove the following Theorem with the key analytic estimates \eqref{e:coeffnorm} and \eqref{e:sublevel} as the departure point.
\begin{thm}[Stein-Wainger]\label{t:SW0}
For each $1<p < \infty$,
\begin{align*}
\| C_{d,D} f \|_{L^p(\mathbb{R}^D)} \lesssim_{d,D,p} \| f\|_{L^p(\mathbb{R}^D)}.
\end{align*}
\end{thm}

\subsection{Using The Estimates }
Since $K$ is a normalized Calder\'{o}n-Zygmund kernel, by \cite[\S 13]{FatS}
there exist a collection of mean zero $\mathcal{C}^1$ functions $\{\psi_j\}$ so that
\begin{align}\label{e:psi}
|\psi_j(x) | \cdot 2^{Dj} + |\nabla \psi_j(x)| \cdot 2^{D(j+1)} \leq C
\end{align}
for each $j$, so that each $\psi_j$ is supported in $\{ 2^{j-2} \leq |x| \leq 2^j \}$, and so that 
\begin{align*}
K(x) = \sum_j \psi_j(x), \; \; \; x \neq 0.
\end{align*}

For $P_\lambda(t) = \sum_\alpha \lambda_\alpha t^\alpha$, decompose
\begin{align*}
&\int  f(x-t) K(t) e(P_\lambda(t)) \ dt \\
& \qquad = \sum_{j > j(\lambda)} \int f(x-t) \psi_j(t) e(P_\lambda(t)) \ dt + \sum_{j \leq j(\lambda)} \int f(x-t) \psi_j(t) e(P_\lambda(t)) \ dt \\
& \qquad \qquad = \sum_{l \geq 1} \sum_{j > j(\lambda)} \int f(x-t) \psi_j(t) e(P_\lambda(t)) \ dt \cdot \mathbf{1}_{ \| P_\lambda(2^j \cdot) \| \sim 2^l } \\
& \qquad \qquad \qquad + \sum_{j \leq j(\lambda)} \int f(x-t) \psi_j(t) \ dt + \sum_{j \leq j(\lambda)} O\big( \int |f(x-t)| |\psi_j(t)| |P_\lambda(t)| \ dt \big),
\end{align*}
see \eqref{e:sim}. Above, $j(\lambda)$ is the maximal integer, $j$, so that
\begin{align}\label{e:comp}
\| P_\lambda(2^j\cdot )\| = \sum_\alpha 2^{j|\alpha}| \cdot |\lambda_\alpha| \leq 1,
\end{align}

In particular, we may bound
\begin{equation}\label{e:A}
\begin{split}
&|\int  f(x-t) K(t) e(P(t)) \ dt|\\
& \qquad \leq \sum_{l \geq 1} \ \sup_P |\int f(x-t) \psi_{j(\lambda,l)}(t) e(P(t)) \ dt| + T^* f(x)+ M_{HL} f(x) \\
& \qquad \qquad =: \sum_{l \geq 1} \mathcal{A}_l f(x) + T^* f(x) + M_{HL} f(x),
\end{split}
\end{equation}
where $T^*f$ is a maximally truncated singular integral, and $j(\lambda,l)$ is defined to be the unique $j$ (if it exists) so that
\begin{align*}
\| P(2^j \cdot) \| \sim 2^l,
\end{align*}
see \eqref{e:comp}; if no such $j$ exists, set $\psi_{j(\lambda,l)} = 0$.

We will prove
\begin{proposition}\label{p:SW0}
For each $1<p< \infty$, there exists an absolute $c_p>0$ so that
\begin{align*}
\| \mathcal{A}_l f \|_{L^p(\mathbb{R}^D)} \lesssim 2^{-c_p l} \|f\|_{L^p(\mathbb{R}^D)}.
\end{align*}
\end{proposition}

Theorem \ref{t:SW0} then follows from a summation over $l \geq 1$.

By interpolation, it suffices to establish Proposition \ref{p:SW0} at the $L^2$ level. If we linearize the supremum, we may express $\mathcal{A}_l$ as an integral operator,
\begin{align*}
\mathcal{A}_l f(x) = \int K(x,t) f(t) \ dt
\end{align*}
with kernel
\begin{align*}
K(x,t) = e(P_{\lambda}(x-t)) \psi_{j(\lambda,l)}(x-t)
\end{align*}
where importantly $\lambda = \lambda(x)$ varies measurably with $x$, see \eqref{e:poly}.

By the Kolmogorov-Seliverstov method of $TT^*$, it suffices to show that
\begin{align*}
\| \int \mathcal{K}(x,y) f(y) \ dy \|_{L^2(\mathbb{R}^D)} \lesssim 2^{-cl} \|f\|_{L^2(\mathbb{R}^D)},
\end{align*}
where
\begin{align*}
\mathcal{K}(x,y) &= \int K(x,t) \overline{K(y,t)} \ dt \\
& \qquad = \int e(P_\lambda(x-t) - P_\mu(y-t)) \psi_k( x -t) \psi_r(y-t) \ dt
\end{align*}
and $\lambda = \lambda(x), \mu = \mu(y)$, and $k = k(\lambda,l)$, $r = r(\mu,l)$. Indeed 
\begin{equation}\label{e:TT*use}
\begin{split}
    \| \mathcal{A}_l f\|_{L^2(\mathbb{R}^D)} &= \| \int K(\cdot,t) f(t) \ dt \|_{L^2(\mathbb{R}^D)} \\
    & \qquad \leq \| K \|_{\text{Ker}(\mathbb{R}^D)} \cdot \| f\|_{L^2(\mathbb{R}^D)} \\
    & \qquad \qquad = \| \mathcal{K} \|_{\text{Ker}(\mathbb{R}^D)}^{1/2} \cdot \|f\|_{L^2(\mathbb{R}^D)} \\
    & \qquad \qquad \qquad \lesssim 2^{-c/2 l} \cdot \|f\|_{L^2(\mathbb{R}^D)}.
\end{split}
\end{equation}
Above, we have used the operator norm
\begin{align*}
    \| K \|_{\text{Ker}(\mathbb{R}^D)} := \sup_{\| f \|_{L^2(\mathbb{R}^D)} = 1} \| \int K(x,t) f(t) \ dt \|_{L^2(\mathbb{R}^D)}.
\end{align*}

The key properties of $\mathcal{K}(x,y)$ are contained in the following Lemma.
\begin{lemma}\label{l:SW0}
There exists an absolute constant $c > 0$ so that
\begin{align*}
|\mathcal{K}(x,y)| &\lesssim 2^{-k D - cl} \mathbf{1}_{|x-y| \lesssim 2^k} + 2^{-rD - cl} \mathbf{1}_{|x-y| \lesssim 2^r } \\
& \qquad  + 2^{-kD} \mathbf{1}_{E(x)}(y) + 2^{-rD} \mathbf{1}_{E(y)}(x),
\end{align*}
where $E(x) \subset \{ |y| \lesssim 2^k \}, \ E(y) \subset \{ |x| \lesssim 2^r \}$ depend only on the identified variables and have \[ 2^{kD} |E(x)| + 2^{rD} |E(y)| \lesssim 2^{-cl}.\]
\end{lemma}

In particular, if we set the \emph{$h$-small set maximal function}
\begin{align*}
M_h f(x) := \sup_k \sup_E \ \frac{1}{2^{kD}} \int_E |f(x-y)| \ dy,
\end{align*}
where the supremum is over
\begin{align*}
E \subset \{|x| \lesssim 2^k\} : |E| \lesssim  2^{kD-h},
\end{align*}
then by interpolating the pointwise inequality $M_h f \lesssim M_{HL} f$ against the $L^\infty$ bound,
\[ \| M_h f \|_{L^{\infty}(\mathbb{R}^D)} \lesssim 2^{-h} \| f\|_{L^\infty(\mathbb{R}^D)} \]
one deduces that $M_h$ is a contraction on each $L^p$ space, and we may bound, for an appropriate $\| g \|_{L^2(\mathbb{R}^D)} = 1$,
\begin{align*}
&\| \int \mathcal{K}(x,y) f(y) \ dy \|_{L^2(\mathbb{R}^D)} \leq \int |g(x)| |\mathcal{K}(x,y)| |f(y)| \ dx dy \\
& \qquad \lesssim 2^{-cl} \int M_{HL} f(x) |g(x)| \ dx + \int M_{cl} f(x) |g(x)| \ dx + \int |f(y)| M_{cl} g(y) \ dy \\
& \qquad \qquad \lesssim 2^{-c/2 l} \| f \|_{L^2(\mathbb{R}^D)}.
\end{align*}

It remains only to prove Lemma \ref{l:SW0}, which follows from the following estimate, after changing variables appropriately, $v = x-y$, and using symmetry to reduce to the case where $ r \leq k$,

\begin{equation}
\label{e:kerbound}
\begin{split}
&|\int e(P_\lambda(v-2^{-r_0} t) - P_\mu(-t)) \psi( v -2^{-r_0} t) \psi(-t) \ dt| \\
& \qquad \qquad \qquad \lesssim 2^{-cl} + \mathbf{1}_{E_\lambda}(v), \; \; \; \; \; \;  r_0 \geq 0
\end{split}
\end{equation}
where $E_\lambda \subset \{ |v| \lesssim 1\}$ depends only on $\lambda$ and has $|E_{\lambda}| \lesssim 2^{-cl}$, and we think of $r_0 = k-r$.

There are two cases. If $2^{-r_0} \leq \eta \ll 1$, then
\[ t \mapsto P_{\lambda}(v-2^{-r_0} t) - P_\mu(-t) = P_{\lambda}(v) + \big( (P_{\lambda}(v-2^{-r_0} t) - P_\lambda(v) ) - P_\mu(-t) \big)
\]
has a coefficient norm that is $\gtrsim 2^l$, as the coefficient norm of
\[ t \mapsto P_{\lambda}(v-2^{-r_0} t) - P_\lambda(v) \]
is $O(2^{l-r_0})$. By \eqref{e:coeffnorm}, the bound \eqref{e:kerbound} holds. In the other case, we can bound the coefficient norm of 
\[ t \mapsto P_{\lambda}(v-2^{-r_0} t) - P_\lambda(-t) \]
from below by the coefficient norm of the linear terms in the above difference; the presence of these linear terms is due to the lack of linear terms in $P_\mu$. In particular
\begin{align*}
\| P_\lambda(v-2^{-r_0} t) - P_\mu(-t) \| &\geq \| \sum_{j=1}^D \Big( \sum_{\alpha > e_j} \lambda_\alpha \alpha_j v^{\alpha - e_j} \Big) t^{e_j} \| \\
& \qquad \geq \sum_{j=1}^D \Big| \sum_{\alpha > e_j} \lambda_\alpha \alpha_j v^{\alpha - e_j} \Big| =: \sum_{j=1}^D |P_j(v)|;
\end{align*}
by another application of \eqref{e:coeffnorm}, it suffices to show that
\begin{align*}
|\{ |v| \lesssim 1 : \sum_{j=1}^D |P_j(v)| \lesssim 2^{c_0 l} \}| \lesssim 2^{-c_0 l}
\end{align*}
for some $c_0 > 0$. But his just follows from \eqref{e:sublevel}:
\begin{align*}
|\{ |v| \lesssim 1 : |P_j(v)| \lesssim 2^{c_0 l} \}| \lesssim \big( \frac{2^{c_0 l}}{ \|P_j \|} \big)^\theta \lesssim 2^{\theta (c_0 -1)l}
\end{align*}
for some $\theta > 0$ (which can be taken to be $1/d$). \\
\medskip

\subsubsection{Stein-Wainger: Continuous Summary and Discrete Preliminaries}
Aside from \eqref{e:sublevel}, which follows directly from \eqref{e:coeffnorm}, the techniques needed to establish Theorem \ref{t:SW0} are fairly modest:
\begin{itemize}
    \item The Hardy-Littlewood maximal function, $M_{HL}$;
    \item Maximally truncated singular integrals;
    \item $TT^*$ arguments;
    \item Interpolation.
\end{itemize}

The clarity of this scheme suggests that a similar approach should extend to the discrete situation, namely to Theorem \ref{t:main}. As is characteristic of the field, our analysis requires more delicacy as the positive Hardy-Littlewood maximal function is insensitive to destructive interference arising from arithmetic considerations. This difficulty can be partially resolved by earlier work \cite{BK11,BK12}, and by applying further $TT^*$ arguments, matters to formulating an arithmetic version of the Stein-Wainger argument.

Accordingly, we apply a discrete analogue of \eqref{e:coeffnorm} to derive the analogue of the crucial sublevel bound \eqref{e:sublevel}.

\section{Exponential Sums and Sublevel Estimates}\label{s:exp}
In this section we introduce the relevant analogues of \eqref{e:coeffnorm} and \eqref{e:sublevel} in the discrete context. As this result has been appeared \cite{Tblog}, we defer its proof to Appendix \ref{s:app}.

We begin by defining scale-dependent coefficient norms in full generality:

\begin{definition}
Suppose $P(x) = \sum_\alpha \lambda_\alpha x^\alpha \in \mathbb{R}[x_1,\dots,x_D]$, that $R_i \geq 1$, and set
\begin{align*}
\vec{R} = (R_1,\dots,R_D).
\end{align*}
Let $s_0$ denote the minimal $s$ so that either there exists some $Q \leq 2^s, \ s \geq 1$ so that
\begin{align*}
\sum_\alpha  \| Q \lambda_\alpha \|_{\mathbb{T}} \cdot \vec{R}^{\alpha} \leq 2^s,
\end{align*}
or
\begin{align*}
\sum_\alpha \| \lambda_\alpha \|_{\mathbb{T}} \cdot  \vec{R}^{\alpha}  \leq 2^s
\end{align*}
if $s \leq 0$.

Then, define the \emph{coefficient norm} of $P$ at scale $\vec{R}$,
\begin{align*}
N_{\vec{R}}(P) := 2^{s_0}.
\end{align*}
\end{definition}

If $\vec{R} = (R,\dots,R)$ is cubic, we will abbreviate
\begin{align*}
N_{\vec{R}}(P) = N_R(P).
\end{align*}

\begin{example}\label{ex:poly}
Suppose that $\lambda_\alpha = \frac{A_\alpha}{Q} \in \mathbb{Q}$ where $(\{ A_\alpha\}, Q) = 1$ are reduced. If $R_i \geq Q^{1+\delta}$ where $1 > \delta > 0$ is small, then \[ N_{\vec{R}}(P) \gtrsim Q^{\delta}.\]
\begin{proof}[Reason]
Suppose that $N_{\vec{R}}(P) \leq 2^t$, where $2^t < \frac{Q}{10}$ (say). This means that there exists $Q_1 \leq 2^t$ so that
\begin{align*}
\sum_\alpha \| \frac{A_\alpha}{Q} Q_1 \|_{\mathbb{T}} \cdot \vec{R}^{\alpha} \leq 2^t;
\end{align*}
since $(\{ A_\alpha\},Q)$ are reduced, there must be some $\alpha_0$ so that $\frac{A_{\alpha_0}}{Q} Q_1 \not \equiv 0 \mod 1$, leading to the lower bound
\begin{align*}
Q^{-1} \leq \| \frac{A_{\alpha_0}}{Q} Q_1 \|_{\mathbb{T}},
\end{align*}
and thus
\begin{align*}
\min_i \frac{R_i}{Q} \leq Q^{-1} \vec{R}^{\alpha_0} \leq \| \frac{A_{\alpha_0}}{Q} Q_1 \|_{\mathbb{T}} \cdot \vec{R}^{\alpha_0} \leq 2^t,
\end{align*}
so $2^t \geq Q^\delta$.
\end{proof}
\end{example}

The significance of the quantity $N_{\vec{R}}(P)$ is that it controls various exponential sums, analogous to the way the coefficient norms control oscillatory integrals in Euclidean space; note the trivial upper bound
\begin{align}\label{e:triv}
N_{\vec{R}}(P) \leq \sum_\alpha \vec{R}^{\alpha},
\end{align}
and the multiplicative property
\begin{align*}
N_{\vec{R}}(P) \leq 2^{\lceil \log_2 k \rceil} \cdot N_{\vec{R}}(k \cdot P)\; \; \; \; k \geq 1
\end{align*}
which is occasionally sharp, as can seen for polynomials with coefficients in $\mathbb{Z}/k \mathbb{Z}$. To see this, just observe that whenever $N_{\vec{R}}(kP) = 2^s$, there exists some multiple of $k$, $Qk$, with $Q \leq 2^s$, so that
\begin{align*}
\sum_\alpha \| Q k \lambda_\alpha \|_{\mathbb{T}} \cdot \vec{R}^{\alpha} \leq 2^s \leq 2^{\lceil \log_2 k \rceil} 2^s.
\end{align*}

Finally, we record the following convexity lemma concerning the coefficient norms.
\begin{lemma}\label{l:conv}
For any $P$, for each $s \geq 1$ the set
\begin{align*}
\{ k : N_{2^k}(P) = 2^s \}
\end{align*}
is an interval. When $s \leq 0$, there exists at most one $k$ so that $N_{2^k}(P) = 2^s$.
\end{lemma}
\begin{proof}
The case where $s \leq 0$ is clear, so we consider the more interesting case where $s \geq 1$. We need to show that if
\[ N_{2^k}(P) = N_{2^h}(P) = 2^s, \; \; \; k \leq h,\]
then $N_{2^{k'}}(P) = 2^s$ as well, for all $k \leq k' \leq h$. Since $N_{2^k}(P) = 2^s$, for any $Q \leq 2^{s-1}$,
\begin{align*}
2^{s-1} < \sum_{\alpha} \| Q \lambda_\alpha \|_{\mathbb{T}}  \cdot  2^{k |\alpha|} \leq \sum_{\alpha} \| Q \lambda_\alpha \|_{\mathbb{T}} \cdot  2^{k' |\alpha|},
\end{align*}
which says that $N_{2^{k'}}(P) \geq 2^s$; since $N_{2^h}(P) = 2^s$, there exist some $Q' \leq 2^s$ so that
\begin{align*}
2^s \geq \sum_\alpha \| Q' \lambda_\alpha \|_{\mathbb{T}}  \cdot  2^{h|\alpha|} \geq \sum_\alpha \| Q' \lambda_\alpha \|_{\mathbb{T}}  \cdot  2^{k'|\alpha|},
\end{align*}
for the reverse inequality.
\end{proof}

The key property of the coefficient norm is captured by the following Theorem, see \cite[Proposition 8]{Tblog}; a complete proof can be found in Appendix \ref{s:app} below.

\begin{thm}[Coefficient Norms Control Exponential Sums]\label{p:coeffnorm}
There exists an absolute $\theta = \theta_{d,D} > 0$ so that the following holds whenever $N_{\vec{R}}(P) = 2^s \geq 2$:

For every multi-dimensional arithmetic progression,
\[ \vec{\mathcal{P}} = \mathcal{P}_1 \times \dots \times \mathcal{P}_D, \; \; \; \mathcal{P}_i \subset [R_i] \]

\begin{align*}    \Big| \frac{1}{|\vec{R}|} \sum_{n \in \mathcal{P}} e(k \cdot P(n)) \phi(n) \Big| \lesssim \big( \frac{|k|}{2^s} \big)^{\theta} + \sum_i R_i^{-\theta} , \end{align*}
 whenever
\begin{align*}
    \| \phi \|_{\ell^\infty} + \sum_{j=1}^D R_j \cdot \| \phi - \phi(\cdot -  {e_j}) \|_{\ell^\infty} \leq 1.
\end{align*}

If $N_{\vec{R}}(k \cdot P) = 2^s \leq 1$, then
\begin{align*}
\frac{1}{|\vec{R}|} \sum_{n \in \mathcal{P}} e(k \cdot P(n)) \phi(n) = \frac{1}{|\vec{R}|} \sum_{n \in \mathcal{P}}  \phi(n) +O(N_{\vec{R}}(k \cdot P)).
\end{align*}
\end{thm}
The second point just follows from the mean-value theorem; the content of Theorem \ref{p:coeffnorm} concerns the case where the coefficient norm is large; the condition on the amplitudes, $\phi$, is standard, as one can quickly reduce to the case of constant amplitudes, see for instance \cite[\S A.1]{MSZK}. 

\subsection{Sublevel Estimates}
In the Euclidean setting, the sublevel estimate \eqref{e:sublevel} follows directly from the coefficient-norm bound \eqref{e:coeffnorm}; in our present context, we can use Theorem \ref{p:coeffnorm} to derive appropriate sublevel estimates as well. 

The following lemmas will be used to bound the percentage of time a polynomial can cluster extremely close to cyclic subgroups with small denominators; these fill the role of \eqref{e:sublevel}.

Our first lemma will be used in the case where the polynomial in question has a fairly large coefficient norm.

\begin{lemma}[Non-Concentration for Polynomials with Small Coefficient Norms] \label{c:nonconc}
Suppose $A \leq N_R(P)^{\theta}$, where $N_{R}(P) \geq 2$ and $\theta = \theta_{d,D} > 0$ is sufficiently small. There exists a $C= C_{d,D}$ so that whenever $B \geq 100$ (say)
\begin{align}\label{e:sub11}
    |\{ |v| \leq R : \min_{q \leq A} \|P(v) q\|_{\mathbb{T}} \leq B^{-1} \}| \lesssim R^D \cdot  A \cdot \Big( N_R(P)^{-\theta} + B^{-\theta} +R^{-\theta} \Big) 
\end{align}
\end{lemma}
\begin{proof}
By the union bound, it suffices to consider a single $q \leq A$, but without the factor of $A$ on \eqref{e:sub11}. Suppose that $N_R(P) = 2^s$. Assume, as we may, that $B$ is an integer. We dominate the indicator function by a Fej\'{e}r kernel and bound
\begin{align*}
&\sum_{|v| \leq R} \mathbf{1}_{\| \beta \|_{\mathbb{T}} \leq B^{-1}}(P(v)q) \\
& \qquad \leq \sum_{|v| \leq R} \Big( \sum_k \mu_B(k) e(k \cdot P(v) q) \Big),
\end{align*}
see \eqref{e:mu}. In particular, if we set $\vec{R} = [R]^D \times [B]$, and define 
\begin{align*}
P_0(v,k) := k \cdot P(v) q \in \mathbb{R}[v_1,\dots,v_D,k]
\end{align*}
to be a polynomial of $D+1$ many variables, it suffices to show that 
\begin{align*}
\Big| \frac{1}{R^D} \sum_{|v| \leq R, \ k} \mu_B(k) e(P_0(v,k)) \Big| \lesssim \Big( N_R(P)^{-\theta} + B^{-\theta} +R^{-\theta} \Big).
\end{align*}

If $N_{\vec{R}}(P_0) \geq 2^{s-1}/A \gtrsim 2^{s(1-\theta)}$, the bound is clear, so assume otherwise.

This says that there exist $Q_1 <  2^{s-1}/A $ so that 
\begin{align*}
\sum_\alpha \| Q_1 (q \lambda_\alpha) \|_{\mathbb{T}} \cdot R^{|\alpha|} \leq B^{-1} \cdot 2^{s-1};
\end{align*}
let $Q_0 := Q_1 q < 2^{s-1}$ be minimal subject to this constraint. 

On the other hand, since $N_R(P) = 2^s$, we know that for every $Q \leq 2^{s-1}$
\begin{align*}
\sum_\alpha \| Q \lambda_\alpha \|_{\mathbb{T}} \cdot R^{|\alpha|} > 2^{s-1};
\end{align*}
specializing to $Q = Q_0$ yields the contradiction.
\end{proof}

The following sub-level set estimate will serve as a substitute to Lemma \ref{c:nonconc}
when $P$ has a very large coefficient norm: $B = N_R(P)^{O(1)}$. 
\begin{lemma}[Non-Concentration Estimate for Polynomials with Large Coefficient Norms]\label{c:sublevel1}
Suppose that $N_R(P) \geq R^{\eta}$ where $0 < \eta \ll 1$ is bounded away from zero, and that $0 <\kappa \ll \eta$ is sufficiently small.

Then there exists some absolute $\kappa_0 = \kappa_0(\eta) > 0$ so that
\begin{align*}
|\{ |v| \leq R : \min_{q \leq R^\kappa} \|P(v) q\|_{\mathbb{T}} \leq R^{\kappa - 1} \}| \lesssim R^{D - \kappa_0}.
\end{align*}
\end{lemma}
\begin{proof}
We use a union bound; thus, it suffices exhibit $\kappa_0 > 0$ so that for each $q \leq R^{\kappa}$,
\begin{align*}
|\{ |v| \leq R : \|P(v) q\|_{\mathbb{T}} \leq R^{\kappa - 1} \}| \lesssim R^{D - \kappa_0 - \kappa}.
\end{align*}
By another Fej\'{e}r kernel argument, it suffices to exhibit a $\kappa_0 > 0$ so that
\begin{align}\label{e:sum00}
\left|
\frac{1}{R^D} \sum_{|v| \leq R, n} \mu_{R^{1-\kappa}}(n) \cdot e(P_0(v,n)) \right| \lesssim R^{-\kappa_0 - \kappa},
\end{align}
where
\begin{align*}
P_0(v,n) = \sum_\alpha (q \cdot \lambda_\alpha) \cdot ( n v^\alpha ) \in \mathbb{R}[v_1,\dots,v_D,n]
\end{align*}
is a polynomial of $D+1$ variables. Set $\vec{R} = (R,\dots,R,R^{1-\kappa})$; we claim that $N_{\vec{R}}(P_0) \geq R^{\eta/2}$; this allows us to bound
\begin{align*}
\eqref{e:sum00} \leq R^\kappa \cdot ( R^{-\theta \eta} + R^{-\theta} ),
\end{align*}
which would yield the result.

So, suppose otherwise, and extract a minimal $Q_1 \leq R^{\eta/2}$ so that
\begin{align*}
\sum_\alpha \| (Q_1 q) \lambda_\alpha \|_{\mathbb{T}} {R}^{|\alpha| + 1 -\kappa} \leq R^{\eta/2}.
\end{align*}
On the other hand, since $N_R(P) \geq R^{\eta}$, for every $Q \leq \frac{R^{\eta}}{2}$ 
\begin{align*}
\sum_\alpha \| Q \lambda_\alpha \|_{\mathbb{T}} R^{|\alpha|} > R^{\eta};
\end{align*}
the contradiction arises by specializing $Q = Q_1 q \leq R^{\eta/2 + \kappa} \ll R^{\eta}$.
\end{proof}

In what follows, we will use the machinery developed above to prove Theorem \ref{t:main}.

\section{The Discrete Stein-Wainger Operator}\label{s:carl}
Regarding $\| K \|_{\text{CZ}(\mathbb{R}^D)}$ as given, let $\{ \psi_j \}$ be as in \eqref{e:psi}.

We introduce a large real parameter, \begin{align}\label{e:A0}
    A_0 = A_0(d,D,p)
\end{align} which we are free to adjust upwards finitely many times as needed.

Define
\begin{align*}
\mathcal{C}_{d,D} f(x) := \sup_{P \in \mathscr{P}_{d,D}, \ k_0} | \sum_{m \in \mathbb{Z}^D} \sum_{k=1}^{k_0}  f(x-m) \psi_k(m) e(P(m)) |;
\end{align*}
Theorem \ref{t:main} follows directly from the following result after an argument with the Hardy-Littlewood maximal function.
\begin{thm}\label{t:key}
For each $1 < p < \infty$ and $d,D$,
\begin{align*}
    \| \mathcal{C}_{d,D} f\|_{\ell^p(\mathbb{Z}^D)} \lesssim \|f\|_{\ell^p(\mathbb{Z}^D)}
\end{align*}
\end{thm}

By Proposition \ref{p:triangle} below, our attention will be focused on the ``oscillatory" truncated singular integrals
\begin{equation}\label{e:As}
\begin{split}
&\mathscr{A}_s f(x) \\
&:= \sup_{k_0 \geq 2^{s/A_0}} \Big| \sum_{k = 2^{s/A_0}}^{k_0} \sum_{m} \psi_k(m) e(P_{\lambda(x)}(m)) f(x-m) \cdot \mathbf{1}_{n : N_{2^k}(P_{\lambda(n)}) = 2^{ s}}(x) \Big|.
\end{split}
\end{equation}
 for $A_0$ as in \eqref{e:A0} a sufficiently large absolute constant, and $\lambda:\mathbb{Z}^D \to [0,1]^{|\Gamma|}$ an appropriate linearizing function. Note that for each $s \geq 1$, the set
\begin{align*}
\{ k : N_{2^k}(P) = 2^s \}
\end{align*}
is an interval, see Lemma \ref{l:conv}.

We present this reduction in the form of a proposition.

\begin{proposition}\label{p:triangle}
The following pointwise bound holds:
\begin{align}\label{e:sum}
\mathcal{C}_{d,D} f \lesssim \sum_{s \geq 1} \mathscr{A}_s f + \mathcal{E} f + H^*f + M_{HL} f,
\end{align}
where $H^*f$ is a maximally truncated singular integral, $\mathscr{A}_s$ are as in \eqref{e:As}, and
\begin{align*}
\mathcal{E} f = \sum_{k \geq 1} \mathcal{E}_k f
\end{align*}
is a sum of single scale operators with
\begin{align}\label{e:Ek}
\| \mathcal{E}_k f \|_{\ell^p(\mathbb{Z}^D)} \lesssim k^{-2 } \cdot \| f\|_{\ell^p(\mathbb{Z}^D)}.
\end{align}
\end{proposition}

The proof of Proposition \ref{p:triangle} will take up the early part of this section, with \eqref{e:Ek} being the crucial point. We accordingly defer the estimate \eqref{e:Ek} to Proposition \ref{p:errors} below.
\begin{proof}[Proof of Proposition \ref{p:triangle} Assuming \eqref{e:Ek}]

With $\lambda(x)$ an appropriate linearizing function, foliate
\begin{align*}
&|\sum_{k=1}^{k_0} \ \sum_{s \in \mathbb{Z}} \sum_{m \in \mathbb{Z}^D} f(x-m) \psi_k(m) e(P_{\lambda(x)}(m))| \\
& \qquad = |\sum_{k=1}^{k_0} \ \sum_{s \in \mathbb{Z}}  \sum_{m \in \mathbb{Z}^D}  f(x-m) \psi_k(m) e(P_{\lambda(x)}(m))
\cdot \mathbf{1}_{ n :  N_{2^k}(P_{\lambda(n)}) \sim 2^s}(x) |
\end{align*}
according to the size of the pertaining coefficient norms.

We bound the foregoing above by the sum of two terms, which we will address individually:
the stationary component
\begin{align}\label{e:stat1}
|\sum_{k=1}^{k_0} \  \sum_{m \in \mathbb{Z}^D}  f(x-m) \psi_k(m) e(P_{\lambda(x)}(m))
\cdot \mathbf{1}_{ n :  N_{2^k}(P_{\lambda(n)}) \leq 1}(x) |;
\end{align}
and
the oscillatory component
\begin{align}\label{e:osc1}
    |\sum_{k=1}^{k_0} \ \sum_{s \geq 1}  \sum_{m \in \mathbb{Z}^D}  f(x-m) \psi_k(m) e(P_{\lambda(x)}(m))
\cdot \mathbf{1}_{ n :  N_{2^k}(P_{\lambda(n)}) \sim 2^s}(x) |.
\end{align}

We begin by bounding \eqref{e:stat1}. Since the coefficient norm of $P_{\lambda(x)}$ is small in this case, we just Taylor expand the phase:
\begin{align*}
\eqref{e:stat1}
&\leq
| \sum_{k=1}^{k_0} \ \sum_{m \in \mathbb{Z}^D} f(x-m) \psi_k(m) 
\cdot \mathbf{1}_{ n :  N_{2^k}(P_{\lambda(n)}) \leq 1}(x) | \\
& \qquad  + 
|\sum_{k=1}^{k_0} \ \sum_{m \in \mathbb{Z}^D}   f(x-m) \psi_k(m) \big( e(P_{\lambda(x)}(m)) - 1 \big)
\cdot \mathbf{1}_{ n :  N_{2^k}(P_{\lambda(n)}) \leq 1}(x) | \\
& =: H^*f(x) + O\Big(
| \sum_{k=1}^{\infty} \sum_{s \leq 0} \sum_{m \in \mathbb{Z}^D} |f(x-m)| |\psi_k(m)| |P_{\lambda(x)}(m)|
\cdot \mathbf{1}_{ n :  N_{2^k}(P_{\lambda(n)}) \sim 2^s}(x) \Big) | \\
& \qquad \lesssim H^*f(x) + M_{HL} f(x) \cdot \Big( \sum_{k=1}^{\infty} \sum_{s \leq 0} 2^s  \mathbf{1}_{ n :  N_{2^k}(P_{\lambda(n)}) \sim 2^s}(x) \Big) \\
& \qquad \qquad \lesssim H^*f(x) + M_{HL} f(x)
\end{align*}
where $H^*$ is a truncated singular integral, as 
\begin{align*}
\{ k : N_{2^k}(P) \leq 1 \}
\end{align*}
is an interval, see Lemma \ref{l:conv}, and we crucially used that for each $s\leq 0$, there is at most one $k$ so that
\begin{align*}
N_{2^{k}}(P) = 2^s,
\end{align*}
see Lemma \ref{l:conv}. This concludes the estimate of \eqref{e:stat1}.

We now address \eqref{e:osc1}. To do so, bound
\begin{align*}
\eqref{e:osc1} 
& \leq 
|\sum_{k=1}^{k_0} \ \sum_{s : 2^s \leq k^{A_0}} \sum_{m \in \mathbb{Z}^D} f(x-m) \psi_k(m) e(P_{\lambda(x)}(m))
\cdot \mathbf{1}_{ n :  N_{2^k}(P_{\lambda(n)}) \sim 2^s }(x)| \\
& \qquad + 
|\sum_{k=1}^{k_0} \ \sum_{s : 2^s > k^{A_0}} \sum_{m \in \mathbb{Z}^D } f(x-m) \psi_k(m) e(P_{\lambda(x)}(m))
\cdot \mathbf{1}_{ n :  N_{2^k}(P_{\lambda(n)}) \sim 2^s }(x)| \\
& \qquad \qquad \leq 
\sum_{s \geq 1} | \sum_{k \geq 2^{s/A_0} }^{k_0} \ \sum_{m \in \mathbb{Z}^D } f(x-m) \psi_k(m) e(P_{\lambda(x)}(m))
\cdot \mathbf{1}_{ n :  N_{2^k}(P_{\lambda(n)}) \sim 2^s }(x) | \\
& \qquad \qquad  \qquad + \sum_{k \geq 1} \ 
\Big| \sum_{m \in \mathbb{Z}^D} f(x-m) \psi_k(m) e(P_{\lambda(x)}(m))
\cdot \mathbf{1}_{ n :  N_{2^k}(P_{\lambda(x)}) \geq k^{A_0} }(x) \Big| \\
& \qquad \qquad \qquad \qquad \leq \sum_{s \geq 1} \mathscr{A}_s f(x) + \sum_{k \geq 1} \mathcal{E}_k f(x)
\end{align*}
where the scale $k$ error terms are defined, explicitly,:
\begin{align*}
    \mathcal{E}_k f(x) := \sup_{P : N_{2^k}(P) \geq k^{A_0}} |\sum_{m \neq 0} e(P(m))\psi_k(m) f(x-m) |.
\end{align*}
\end{proof}

In particular, we have reduced the proof of Proposition \ref{p:triangle} to establishing \eqref{e:Ek}. This will be the focus of the following Proposition.

\begin{proposition}\label{p:errors}
There exists an absolute $c = c_{d,D} > 0$ so that for each $1 < p < \infty$
\begin{align*}
\| \mathcal{E}_k f \|_{\ell^p(\mathbb{Z}^D)} \lesssim k^{-c \cdot A_0 \frac{2}{p^*}} \cdot \| f\|_{\ell^p(\mathbb{Z}^D)}, \; \; \; p^* = \max\{p,p'\}.
\end{align*}
\end{proposition}

We prove Proposition \ref{p:errors} by showing that
\begin{align}\label{e:errors}
    &\| \sup_{P : N_{2^k}(P) = 2^s} | \sum_{m} e(P(m)) \psi_k(m) f(x-m)| \|_{\ell^2(\mathbb{Z}^D)}\\
    & \qquad \lesssim ( 2^{-\theta s} + 2^{-\theta k} ) \cdot \| f\|_{\ell^2(\mathbb{Z}^D)} \\
& \qquad \qquad \lesssim 2^{-cs} \| f\|_{\ell^2(\mathbb{Z}^D)}
\end{align}
and interpolating against a trivial single-scale estimate, see \eqref{e:triv} for the final inequality.

We prove \eqref{e:errors} by the Kolmogorov-Seliverstov method of $TT^*$. Regarding $s$ and $k$ as fixed, subject to the constraint $k^{A_0} \leq 2^s \leq 2^{C_{d,D} k}$, see \eqref{e:triv}, it suffices to prove that, uniformly in measurable maps
\begin{align*}
\mathbb{Z}^D \to \{ P \in \mathscr{P}_{d,D} : N_{2^k}(P) = 2^s \}
\end{align*}
the kernel 
\begin{align*}
    K(x,m) := e(P_{\lambda(x)}(x-m)) \psi_k(x-m),
\end{align*}
satisfies 
\begin{align*}
    \| K \|_{\text{Ker}(\mathbb{Z}^D)} := \sup_{\| f \|_{\ell^2(\mathbb{Z}^D)} = 1} \| \sum_m K(\cdot,m) f(m) \|_{\ell^2(\mathbb{Z}^D)} \lesssim 2^{-cs}.
\end{align*}

By arguing as in \eqref{e:TT*use} above, it suffices to instead bound
\begin{align}\label{e:TT*} 
    \| \mathcal{K} \|_{\text{Ker}(\mathbb{Z}^D)} \lesssim 2^{-2cs}
\end{align}
where
\begin{align}\label{e:kernel}
    \mathcal{K}(x,n) = \sum_m e(P_{\lambda(x)}(x-m) - P_{\mu(n)}(n-m)) \psi_k(x-m) \psi_k(n-m)
\end{align}
where
\[ N_{2^k}(P_{\lambda(x)} ) = N_{2^k}(P_{\mu(n)}) = 2^s.\]
The key point is the following Lemma.
\begin{lemma}\label{l:kernel1}
There exists an absolute $c_0 > 0$ so that the following inequality holds pointwise: 
\begin{align*}
   |\mathcal{K}(x,n)| \lesssim 2^{-c_0s - Dk} \cdot \mathbf{1}_{|x-n| \lesssim 2^k} + 2^{-kD}  \cdot \mathbf{1}_{O(x)}(n)
\end{align*}
where 
\begin{align*}
    O(x) \subset \{ |v| \lesssim 2^k\}, \; \; \; |O(x)| \lesssim 2^{-c_0s + kD}
\end{align*} depends only on the $x$ variable.
\end{lemma}
In particular,
\begin{align*}
\sup_x \sum_n |\mathcal{K}(x,n)| \lesssim 2^{-c_0s}, \; \; \; \sup_n \sum_x |\mathcal{K}(x,n)| \lesssim 1,
\end{align*}
so \eqref{e:TT*} follows from Schur's test, with $c = \frac{c_0}{2}$.

\subsection{The Proof of Lemma \ref{l:kernel1}}
With $v = x-n$, $\lambda = \lambda(x)$ and $\mu = \mu(n)$, we need to show that
\begin{align*}
    &    |\mathcal{K}(x,n)| = |\sum_{m} \psi_k(v-m) \psi_k(-m) e(P_\lambda(v-m) - P_\mu(-m)) |\\ & \qquad \lesssim 2^{-kD - \kappa s}  \cdot \mathbf{1}_{|v| \lesssim 2^k} + 2^{-k D} \cdot \mathbf{1}_{\mathcal{U}(\lambda)}(v)
\end{align*}
where $\mathcal{U}(\lambda) \subset [2^k]^D$ has $|\mathcal{U}(\lambda)| \leq 2^{kD - \kappa s}$, and is independent of $\mu$.

The coefficient norm of the phase
\begin{align*}
    m \mapsto P_\lambda(v-m) - P_\mu(-m)
\end{align*}
is bounded below by that of the linear terms in $m$:
\begin{align*}
    m \mapsto \sum_{j=1}^D \big( \sum_\alpha \lambda_\alpha \alpha_j v^{\alpha - e_j} \big) m^{e_j} =: \sum_{j=1}^D P_j(v) \cdot m^{e_j}
\end{align*}

We distinguish between two cases according to the relationship between $s$ and $k$. Below, we let $\eta = \eta_{d,D} \ll 1$ denote a sufficiently small constant.\\

\medskip

\textbf{Case One: $s \leq \eta k$.}\\
Let $\kappa_0 > 0$ be a very small constant, depending on $d,D$. 

Collect
\begin{align*}
    \mathcal{G} := \{ |v| \lesssim 2^k : N_{2^k}\Big( \sum_{j=1}^D P_j(v) m^{e_j} \Big) \geq 2^{\kappa_0 s} \}
\end{align*}
and observe that for $v \in \mathcal{G}$, 
\begin{align*}
    |\sum_{m} \psi_k(v-m) \psi_k(-m) e(P_\lambda(v-m) - P_\mu(-m)) | \lesssim 2^{- \theta \kappa_0 s} + 2^{- \theta k} \lesssim 2^{-\theta \kappa_0s}.
\end{align*}

We set
\begin{align*}
    \mathcal{U}(\lambda) := [2^k]^D \smallsetminus \mathcal{G} \subset \bigcap_{j=1}^D \mathcal{B}_j
    \end{align*}
where
\begin{align*}
    \mathcal{B}_j &:=
    \{ |v| \lesssim 2^k : \min_{q \leq 2^{\kappa_0 s}} \| q P_j(v) \|_{\mathbb{T}} \leq 2^{\kappa_0 s-k} \} 
\end{align*}
In the language of Lemma \ref{c:nonconc}, $A= 2^{\kappa_0 s}$, $B = 2^{k-\kappa_0 s}$, and $R = 2^k$, so by that lemma 
\begin{align*}
|\mathcal{B}_j| \lesssim 2^{kD} \cdot 2^{-cs}.
\end{align*}

\medskip

\textbf{Case Two: $\eta k < s  \lesssim_{d,D} k$.}\\
Collect
\begin{align*}
    \mathcal{G} := \{ |v| \lesssim 2^k : N_{2^k}\Big( \sum_{j=1}^D P_j(v) m^{e_j} \Big) \geq 2^{\kappa k} \}
\end{align*}
and observe that for $v \in \mathcal{G}$, 
\begin{align*}
    |\sum_{m} \psi_k(v-m) \psi_k(-m) e(P_\lambda(v-m) - P_\mu(-m)) | \lesssim 2^{- Dk - \theta \kappa k} \lesssim 2^{-Dk - \theta' s}.
\end{align*}

Collect
\begin{align*}
    \mathcal{U}(\lambda) := [2^k]^D \smallsetminus \mathcal{G} \subset \bigcap_{j=1}^D \mathcal{B}_j
    \end{align*}
where
\begin{align*}
    \mathcal{B}_j &:=
    \{ |v| \lesssim 2^k : \min_{q \leq 2^{\kappa k}} \| q P_j(v) \|_{\mathbb{T}} \leq 2^{(\kappa -1)k} \} 
\end{align*}
By our second sub-level set estimate, Lemma \ref{c:sublevel1}, 
\[ |\mathcal{B}_j| \lesssim 2^{k(D - \kappa_0)} \lesssim 2^{kD - \kappa_1 s},\]
 which completes the proof.

With Proposition \ref{p:triangle} in mind, in the following section we will apply the circle method to approximate our oscillatory operators $\{ \mathscr{A}_s \}$, see \eqref{e:As}, by more tractable family of analytically-defined operators.

\section{Approximations}\label{s:nthry}

We now construct analytic approximates to the multipliers 
\begin{align}\label{31-e:Mj}
m_{j,\lambda}(\beta) := \sum_m \psi_j(m) e(-P_\lambda(m) - \beta \cdot m),
\end{align}
where $\lambda \in [0,1]^{|\Gamma|}$,
\begin{align*}
    P_\lambda(m) = \sum_\alpha \lambda_\alpha m^\alpha,
\end{align*}
see \eqref{e:poly}, and $\beta \in [0,1]^D$.
\begin{align*}
    \beta = (\beta_1,\dots,\beta_D).
\end{align*}

For $(\frac{A}{Q},\frac{B}{Q}) \in \mathbb{Q}^{|\Gamma|} \times \mathbb{Q}^D$, define the complete Gauss sum
\begin{align*}
    S(A/Q,B/Q) = \frac{1}{Q^D} \sum_{r \in (Q)^D} e( - P_{A/Q}(r) - \frac{B}{Q} \cdot r)
\end{align*}
\begin{lemma}
Suppose that $(A,B,Q) = 1$, but $(A,Q) = v > 1$. Then
\[ S(A/Q,B/Q) = 0.\]
\end{lemma}
\begin{proof}
Express $\frac{A_\alpha}{Q} = \frac{a_\alpha}{R}$ and $\frac{B_i}{Q} = \frac{B_i}{Rv}$. Expressing 
\[ r = pR + l\]
we have
\begin{align*}
    P_{A/Q}(r) + \frac{B}{Q} \cdot r \equiv P_{A/Q}(l) + \frac{B}{Rv} \cdot l + \big( \frac{B}{v} \cdot p \big) 
\end{align*}
In particular, since $v$ does not divide at least one of the $B_i$,
\begin{align*}
    S(A/Q,B/Q) &= \frac{1}{R^D} \sum_{l \in (R)^D} e(- P_{A/Q}(l) - \frac{B}{Q} \cdot l) \times \Big( \frac{1}{v^D} \sum_{p \in (v)^D } e( - \frac{B}{v} \cdot p) \Big) \\
& \qquad = 0.
\end{align*}
\end{proof}

Next, define
\[ \Phi_{j,\nu}(\beta) := \int e(-P_\nu(t) - \beta \cdot t) \psi_j(t) \ dt,\]
and
\[ \Phi_{j,\nu}^*(\beta) = \Phi_{j,\nu}(\beta) \cdot \mathbf{1}_{|\nu_\alpha| \leq j^{A_0} 2^{-j |\alpha|}}.\]

With
\begin{align*}
    L_{j,\lambda}^s(\beta) = \sum_{ \frac{A}{Q} : Q \sim 2^s} \sum_{B \in (Q)^D} S(A/Q,B/Q) \Phi_{j, \lambda - A/Q}^*(\beta - B/Q) \chi_s(\beta - B/Q)
\end{align*}
for $\chi_s$ a Schwartz function which satisfies
\begin{align*}
    \mathbf{1}_{|\beta_i| \leq  2^{-2^{10 \rho s}}} \leq \chi_s(\beta) \leq  \mathbf{1}_{|\beta_i| \leq 10 \cdot 2^{-2^{10 \rho s}}}
\end{align*}    
for $\rho  > 0$ an extremely small constant determined below, consolidate
\begin{align*}
    L_{j,\lambda}(\beta) := \sum_{s: 2^s \leq j^{A_0}} L_{j,\lambda}^s(\beta).
\end{align*}

By arguing similarly to \cite{BK11,BK12}, we show that 
\[ \sup_{\lambda} |L_{j,\lambda}^{\vee} * f| \]
well approximates
\[ \sup_{\lambda} |m_{j,\lambda}^{\vee} * f| \]
provided that 
\begin{align}\label{e:Xk}
\lambda \in X_j := \prod_{\alpha} \Big( \bigcup_{q \leq j^{A_0}} \mathbb{Z}/q \mathbb{Z} + O(j^{A_0} 2^{-j|\alpha|}) \Big)
\end{align}

\begin{lemma}
Let $2^{-j} \leq \delta \leq 1$ be a small constant, and suppose that $|\lambda_\alpha - \frac{A_\alpha}{Q}| \leq \delta \cdot 2^{-j(|\alpha| -1)}$ for each $\alpha \in \Gamma$, and that $|\beta_i - \frac{B_i}{Q}| \leq \delta$ for each $1 \leq i \leq D$.

Then
\begin{align*}
    m_{j,\lambda}(\beta) = S(A/Q,B/Q) \Phi_{j,\lambda-A/Q}(\beta - B/Q) + O(Q \delta).
\end{align*}
\end{lemma}
\begin{proof}
With $m = pQ + r$, express
\begin{align*}
    P_{\lambda}(pQ + r) + \beta \cdot (pQ + r) &\equiv P_{A/Q}(r) + B/Q \cdot r \\
& \qquad + \big( P_{\lambda - A/Q}(pQ) + (\beta - B/Q) \cdot pQ \big) + O(Q \delta) \mod 1
\end{align*}
Summing yields
\begin{align*}
    m_{j,\lambda}(\beta) &= \sum_{p,r} \psi_j(pQ) e(- P_\lambda(pQ + r ) - \beta \cdot (pQ + r)) + O(Q \cdot 2^{-j}) \\
    & \qquad = S(A/Q,B/Q) \cdot \sum_p Q^D \psi_j(pQ) \cdot e(-P_{\lambda-A/Q}(pQ) - (\beta -B/Q) \cdot pQ) + O(Q \delta) \\
    & \qquad \qquad = S(A/Q,B/Q) \cdot \Phi_{j,\lambda-A/Q}(\beta - B/Q) + O(Q\delta)
\end{align*}
by a Riemann sum approximation.
\end{proof}

\begin{lemma}
Suppose that $L_{j,\lambda}(\beta) \neq 0$. Then there exists precisely one $(A/Q,B/Q)$ with $Q \leq j^{A_0}$ so that
\begin{align*}
    L_{j,\lambda}(\beta) = L_{j,\lambda}^s(\beta) = S(A/Q,B/Q) \Phi_{j,\lambda - A/Q}(\beta - B/Q) \chi_s(\beta - B/Q)
\end{align*}
if $Q \sim 2^s$.
\end{lemma}
\begin{proof}
Suppose that there exists some $(A'/Q',B'/Q')$ so that 
\[ |\lambda_\alpha - \frac{A_\alpha'}{Q'}| \leq j^{A_0} 2^{-j |\alpha|}\]
If $A'/Q' \neq A/Q$, we would have the following chain of inequalities.
\[ j^{-2A_0} \leq \frac{1}{Q Q'} \leq |\frac{A_\alpha}{Q} - \frac{A_\alpha'}{Q'} | \lesssim j^{A_0} 2^{-2j} \]
Now, if $Q' \sim 2^{s_0}$ with $s_0 > s$, then $S(A'/Q',B'/Q') = 0$, as $A'/Q'$ would not be in reduced form. The only case to check is when there exist $B'/Q' \neq B/Q$ so that
\begin{align*}
    |\beta_i - \frac{B_i}{Q}|, \ |\beta_i - \frac{B'_i}{Q'}| \lesssim 2^{-2^{10 \rho s}}.
\end{align*}
If $B/Q \neq B'/Q'$, then 
\[ |\frac{B_i}{Q} - \frac{B_i'}{Q'}| \geq \frac{1}{QQ'} \approx 2^{-2s},\]
for the desired contradiction.
\end{proof}

\subsection{Major Arcs}
Let $\epsilon_0 = 2^{-10}$, and for $Q \leq 2^{\epsilon_0 j}$, define
\begin{align*}
    \mathfrak{M}_j(A/Q,B/Q) := \{ (\lambda,\beta) \in \mathbb{T}^{|\Gamma|} \times \mathbb{T}^D : |\lambda_\alpha - \frac{A_\alpha}{Q}| \leq 2^{(\epsilon_0 - |\alpha|)j}, \ |\beta_i - \frac{B_i}{Q}| \leq 2^{(\epsilon_0 - 1)j} \},
\end{align*}
and collect
\begin{align*}
    \mathfrak{M}_j = \bigcup_{Q \leq 2^{\epsilon_0 j}} \mathfrak{M}_j(A/Q,B/Q)
\end{align*}

\begin{proposition}
Suppose $\lambda \in X_j$. Then there exists some $c_0 = c_0(d,D)> 0$ so that
\begin{align*}
    |m_{j,\lambda}(\beta) - L_{j,\lambda}(\beta)| \leq 2^{-c_0 \epsilon_0 j}.
\end{align*}
\end{proposition}
\begin{proof}
First suppose that $(\lambda,\beta) \notin \mathfrak{M}_j$. Then for all $Q \leq 2^{\epsilon_0 j}$, 
\begin{align*}
    \sum_{\alpha \in \Gamma} \| Q \lambda_\alpha \|_{\mathbb{T}} \cdot 2^{j|\alpha|} + \sum_{i=1}^D  \| Q \beta_i \|_{\mathbb{T}}  \cdot 2^j > 2^{\epsilon_0 j}.
\end{align*}
So, $|m_{j,\lambda}(\beta)| \lesssim 2^{-\epsilon' \theta j} + 2^{-\theta j}$ for some appropriate $\theta = \theta(d,D) > 0$.

We next observe that for each $Q \leq j^{A_0}$ the Euclidean coefficient norm of the phase 
\begin{align}\label{e:coeffn}
    t \mapsto  P_{\lambda - A/Q}(2^j t) + (\beta - \frac{B}{Q}) \cdot 2^j t
\end{align}
is $\geq 2^{\epsilon_0 j}$, see \eqref{e:Ecn}. Consequently,
\begin{align*}
    |L_{j,\lambda}(\beta)| \lesssim 2^{-\epsilon_0 j \theta }, \; \; \; \theta = 1/d
\end{align*}
by stationary phase estimates, see \eqref{e:coeffnorm}. Above, we used the fact that $L_{j,\lambda}(\beta) = L_{j,\lambda}^s(\beta)$ for some unique $s = s(\lambda,\beta)$, and once again the fact that for each $Q \leq 2^{\epsilon_0 j}$
\begin{align*}
\| \eqref{e:coeffn} \| = \sum_\alpha |\lambda_\alpha - \frac{A_\alpha}{Q}| \cdot 2^{j|\alpha|} + \sum_i |\beta_i - \frac{B_i}{Q}|\cdot 2^j \geq 2^{\epsilon_0 j},
\end{align*}
see \eqref{e:Ecn}.

Next, suppose that $(\lambda,\beta) \in \mathfrak{M}_j(A/Q,B/Q)$ with $Q \sim 2^{s_0} \leq j^{A_0}$.

Since $\lambda \in X_j$, there exists some $\frac{A'}{Q'}$ so that
\begin{align*}
    |\lambda_\alpha - \frac{A'_\alpha}{Q'} | \leq j^{A_0} 2^{-j |\alpha|},
\end{align*}
which forces $A'/Q' = A/Q$, as otherwise one would arrive at the following chain of inequalities for some $\alpha$:
\begin{align*}
j^{-2A_0} \leq |\frac{A_\alpha}{Q} - \frac{A'_\alpha}{Q'}| \leq |\lambda_\alpha - \frac{A_\alpha}{Q}| + |\lambda_\alpha - \frac{A_\alpha'}{Q'}| \leq 2^{(\epsilon_0 -|\alpha|)j} + j^{A_0} 2^{-j|\alpha|},
\end{align*}
Since $\beta \in \mathfrak{M}_j(A/Q,B/Q)$
\begin{align*}
    |\beta_i - B_i/Q| \leq 2^{(\epsilon_0 - 1)j} \ll 2^{-2^{c_0 s}} \ll 2^{-2^{10 \rho s}}
\end{align*}
so
\begin{align*}
    L_{j,\lambda}^{s_0}(\beta) = S(A/Q,B/Q) \Phi_{j,\lambda - A/Q}^*(\beta - B/Q) 
\end{align*} 
while
\begin{align*}
    m_{j,\lambda}(\beta) = S(A/Q,B/Q) \Phi_{j,\lambda- A/Q}^*(\beta - B/Q) + O(2^{(2\epsilon_0 - 1)j})
\end{align*}
where we have recalled the bounds: $Q \leq 2^{\epsilon_0 j}$ and $\delta \leq 2^{(\epsilon_0 -1)j}$.

Finally if $(\lambda,\beta) \in \mathfrak{M}_j(A/Q,B/Q)$ with $j^{A_0} < Q \leq 2^{\epsilon_0 j}$, we would necessarily have $A/Q = A'/Q'$ for some $A'/Q' \in X_j$ (so $Q' \leq j^{A_0})$. This would force $(A,Q) > 1$, and thus
\begin{align*}
    L_{j,\lambda}(\beta), m_{j,\lambda}(\beta) = O(2^{(2\epsilon_0 - 1)j})
\end{align*}
\end{proof}

\begin{lemma}
Set 
\begin{align*}
    \mathcal{E}_{j,\lambda}(\beta) := m_{j,\lambda}(\beta) - L_{j,\lambda}(\beta)
\end{align*}
Then there exists $c_0 > 0$ so that 
\begin{align*}
    \sup_\lambda \| \partial_\lambda^{\alpha_1,\dots,\alpha_m} \mathcal{E}_{j,\lambda}^{\vee} * f \|_{\ell^p(\mathbb{Z}^D)} \lesssim_{d,D} 2^{-c_0 \epsilon_0 \frac{2}{p*} j} \cdot 2^{j(|\alpha_1|+ \dots + |\alpha_m|)} \| f\|_{\ell^p(\mathbb{Z}^D)}
\end{align*}
for each $\alpha_1,\dots,\alpha_m, m \leq |\Gamma|, \ 1 < p < \infty$. 
\end{lemma}
\begin{proof}
The $\ell^2$ estimate without any derivatives was just proven. To handle derivatives on $\ell^2$, just observe that
\begin{align*}
\partial_\lambda^{\alpha_1,\dots,\alpha_m} m_{j,\lambda} = 2^{j(|\alpha_1|+\dots+|\alpha_m|)} m_{j,\lambda}^{\alpha_1,\dots,\alpha_m}
\end{align*}
where $m_{j,\lambda}^{\alpha_1,\dots,\alpha_m}$ is like $m_{j,\lambda}$, except the amplitude $\psi_j(t)$ is replaced by
\begin{align*}
\prod_{i=1}^m \frac{t^{\alpha_i}}{2^{j|\alpha_i|}} \psi_j(t),
\end{align*}
which satisfies all of the same differential inequalities as does $\psi_j$ up to an absolute constant depending on $d,D$; similarly for $L_{j,\lambda}$.

To handle the $\ell^p$ estimates, bound 
\begin{align}\label{e:pointL}
    |\partial_\lambda^{\alpha_1,\dots,\alpha_m} L_{j,\lambda}^{\vee}(x)| \lesssim j^{O_{d,D}(A_0)} \cdot 2^{j(|\alpha_1|+\dots+ |\alpha_m|)} \cdot 2^{-j D} \cdot (1 + 2^{-j}|x|)^{-100},
\end{align}
and 
\begin{align*}
    |\partial_\lambda^{\alpha_1,\dots,\alpha_m} m_{j,\lambda}^{\vee}(x)| \lesssim 2^{j(|\alpha_1|+\dots+ |\alpha_m|)} \cdot |\psi_j(x)|.
\end{align*}
To see \eqref{e:pointL}, we just bound
\begin{align*}
    |L_{j,\lambda}^{\vee}(x)| &\lesssim j^{\epsilon} \cdot \max_{2^s \lesssim j^{A_0}} |(L_{j,\lambda}^s)^{\vee}(x)| \\
    & \qquad \lesssim j^{O_{d,D}(A_0)} \max_{B/Q : Q \leq j^{A_0}, \ 2^s \lesssim j^{A_0}} | \big( \Phi_{j, \lambda - A/Q}^*(\beta - B/Q) \chi_s(\beta - B/Q) \big)^{\vee}(x)| \\
    & \qquad \lesssim j^{O_{d,D}(A_0)} \cdot 2^{-j D} \cdot (1 + 2^{-j}|x|)^{-100},
\end{align*}
where the final estimate follows since the spatial scale of $2^j \gg 2^{2^{10 \rho s}}$ is so large compared to that of $\chi_s$.
\end{proof}

\begin{proposition}
There exists some $c= c(d,D) > 0$ so small that 
\begin{align*}
    \| \sup_{\lambda \in X_j} |\mathcal{E}_{j,\lambda}^{\vee}*f| \|_p \lesssim 2^{-c \epsilon_0 \cdot \frac{2}{p^*} \cdot j} \|f\|_p
\end{align*}
where $p^* := \max\{p,p'\}$.
\end{proposition}
\begin{proof}
By decomposing $X_j$ into $j^{2 \cdot A_0|\Gamma|}$ many boxes of dimensions
\begin{align*}
\{ [2^{-j|\alpha|} ] : \alpha \in \Gamma\},
\end{align*}
it suffices to prove the estimate for a single box.

The proof is by induction on $|\Gamma|$. Thus, let 
\begin{align*}
\mathfrak{P}(R)
\end{align*}
denote the statement that for all $F : Q \times \mathbb{Z}^{D} \to \mathbb{C}$ with $Q$ a box of side-lengths $\{ [L_i] : 1\leq i \leq R\}$ satisfying
\begin{align*}
\sup_\lambda \| \partial^\gamma_\lambda F(\lambda;x) \|_{\ell^p(X)} \leq 100^{|\gamma|} \cdot  \vec{L}^{-\gamma} \cdot a(p)
\end{align*}
where $\max_i \gamma_i \leq 1$, the following estimate holds for some absolute constant $C_R < \infty$:
\begin{align*}
\| \sup_\lambda | F(\lambda;x)|  \|_{\ell^p(X)} \leq C_R \cdot a(p).
\end{align*}
Note that once we have established $\mathfrak{P}(R)$ for $R = |\Gamma|$, we may specialize
\begin{align*}
F(\lambda;x) := \mathcal{E}_{j,\lambda}^{\vee}*f(x),
\end{align*}
where $a(p) = 2^{-c_0 \epsilon_0 \frac{2}{p*}\cdot j} \cdot \| f\|_{\ell^p(\mathbb{Z}^D)}$.

Note that since the statement is translation invariant, we can and will assume that each box is centered at the origin; by dilation invariance we may assume that $L_i = 1$ for each $i$.

When $R = 1$, so $Q = [1]$, the result follows from the pointwise bound
\begin{align*}
|F(\lambda;x)|^p &\lesssim |F(\mu;x)|^p + \Big| \int_{[\mu,\lambda]} \partial_t F(t;x) \ dt \Big|^p \\
& \qquad \leq |F(\mu;x)|^p + \int_{[\mu,\lambda]} |\partial_t F(t;x)|^p \ dt \\
& \qquad \qquad \leq |F(\mu;x)|^p + \int_{[1]} |\partial_t F(t;x)|^p \ dt
\end{align*}
where $\mu \in [1]$ is arbitrary and 
\begin{align*}
\sup_{t \in [1]} \| \partial_t F(t;x) \|_{\ell^p(\mathbb{Z}^D)} \leq a(p).
\end{align*}

For the inductive statement, express $\lambda = (\lambda',\lambda_R)$, and pointwise bound
\begin{align*}
|F(\lambda',\lambda_R;x)|^p &\lesssim \sup_{\lambda'} |F(\lambda',\mu_R;x)|^p + \Big| \int_{[\mu_R,\lambda_R]} \partial_t F(\lambda', t;x) \ dt \Big|^p \\
& \qquad \leq \sup_{\lambda'} |F(\lambda',\mu_R;x)|^p + \int_{[1]} \sup_{\lambda'} |\partial_t F(\lambda',t;x)|^p \ dt \\
& \qquad \qquad \leq \sup_{\lambda'} |F(\lambda',\mu_R;x)|^p + \int_{[1]} \sup_{\lambda'} |\partial_t F(\lambda',t;x)|^p \ dt,
\end{align*}
where $\mu_R \in [1]$ is arbitrary.

But, $\mathfrak{P}(R-1)$ applies to both
\begin{align*}
\lambda' \mapsto F(\lambda',\mu_R;x), \; \; \; \lambda' \mapsto \partial_t F(\lambda',t;x)
\end{align*}
uniformly in $\mu_R,t$, with the same constant $a(p)$, which closes the induction.

\end{proof}

\section{Analytic Estimates}\label{s:finish}
In the previous section, we reduced matters to estimating
\begin{align*}
\sup_{j_0,\lambda} |\sum_{j=1}^{j_0} L_{j,\lambda}^{\vee}*f|.
\end{align*}
We will decompose this maximal function by pigeon-holing in the sizes of the denominators of our rational approximates to $\{\lambda_\alpha\}$.

Define the operator
\begin{align*}
    L_\lambda^{s;j_0} := \sum_{2^{s/A_0} \leq j \leq j_0} L_{j,\lambda}^s,
\end{align*}
and, for any given $\lambda,j_0$ majorize

\begin{align*}
|\sum_{j=1}^{j_0} L_{j,\lambda}^{\vee}*f| & \leq \sum_{s = 1}^{\infty} \Big|\sum^{j_0}_{j : j^{A_0} \geq 2^s} \big( L^s_{j,\lambda}\big)^{\vee}* f| \\
& \qquad \leq \sum_{s =1}^{\infty} \sup_{\mu, j_0} |\big(L^{s;j_0}_\mu\big)^{\vee}*f|.
\end{align*}

In particular, the proof will be complete once we have proven the following proposition.

\begin{proposition}\label{p:mainprop}
There exists an absolute $c = c(d,D,p)$ so that for each $s \geq 1$,
\begin{align*}
    \| \sup_{\lambda,j_0} |(L_\lambda^{s;j_0})^{\vee}*f| \|_{\ell^p(\mathbb{Z}^D)} \lesssim 2^{-c s} \| f\|_{\ell^p(\mathbb{Z}^D)}.
\end{align*}
\end{proposition}

We introduce some further notation to stream-line the proof.

For a bounded multiplier, $m$, define
\begin{align*}
    \mathscr{L}_{s,A/Q}[m](\beta) := \sum_{B \in (Q)^D} S(A/Q,B/Q) m(\beta - B/Q) \chi_s(\beta - B/Q) 
\end{align*}
where $\chi_s$ is as above. We next recall the Ionescu-Wainger exhaustion of the rationals: there exists a function 
\begin{align*}
    h: \mathbb{Q}^D \to 2^{\mathbb{N}}
\end{align*}
so that
\begin{align*}
    h(B/Q) \leq Q
\end{align*}
if $(B,Q) = 1$, and if
\begin{align*}
    \mathcal{U}_s := \{ B/Q : h(B/Q) = 2^s \}
\end{align*}
then the following holds:
\begin{itemize}
    \item $\mathcal{U}_s \subset \{ B'/Q' : Q' \leq 2^{2^{s \rho}} \}$, where $\rho > 0$ is chosen sufficiently small relative to all other parameters introduced; and
    \item If $\chi_s'$ is like $\chi_s$, but is one on its support, then multipliers
    \begin{align*}
    \Pi_{m,s}(\beta) = \sum_{\theta \in \mathcal{U}_s} m(\beta - \theta) \cdot \chi_s'(\beta - \theta)
\end{align*}
satisfy
\begin{align*}
  \| \Pi_{m,s} \|_{M_p(\mathbb{Z}^D)} \lesssim \| m \|_{M_{2r}(\mathbb{R}^D)}
\end{align*}
for any $(2r)' \leq p \leq 2r$, where $M_p(X)$ denotes the multiplier norm
\begin{align*}
    \| M \|_{M_p(X)} := \sup_{\| f\|_{L^p(X)} = 1} \| M^{\vee}*f \|_{L^p(X)}, \; \; \; X = \mathbb{Z}^D, \mathbb{R}^D.
\end{align*}
\end{itemize}
This construction is ultimately due to Tao, \cite{Tao}, building of breathrough work of Ionescu and Wainger \cite{IW} and subsequent refinements \cite{MST1,MSZK}.

Note that we can factor
\begin{align}\label{e:factor}
    \mathscr{L}_{s,A/Q}[m] = \mathscr{L}_{s,A/Q}[1] \cdot \Pi_{m,s},
\end{align}
which is a key point in establishing Lemma \ref{l:borrow}.

We now observe the following identity, which we capture in the following lemma.    
\begin{lemma}\label{l:easymaj}
The following identity holds:
\begin{align*}
    \mathscr{L}_{s,A/Q}[m]^{\vee}(n) &= \sum_{B \in (Q)^D} S(A/Q,B/Q) e(B/Q n) (m \chi_s)^{\vee}(n) \\
    & \qquad = e(-P_{A/Q}(n)) \cdot (m \chi_s)^{\vee}(n)
\end{align*}
and in particular,
\begin{align*}
    \sup_{A/Q} |\mathscr{L}_{s,A/Q}[m]^{\vee}*f(n)| \leq |(m \chi_s)^{\vee}|*|f|(n)
\end{align*}
pointwise.
\end{lemma}
\begin{proof}
The proof is just computation:
\begin{align*}
&\sum_{B \in (Q)^D} S(A/Q,B/Q) \int m(\beta - B/Q) \chi_s(\beta - B/Q) e(\beta n) \ d\beta \\
& \qquad = \sum_{B \in (Q)^D} S(A/Q,B/Q) \cdot e(B/Q n) \cdot (m \chi_s)^{\vee}(n) \\
& \qquad \qquad = \sum_{r \in (Q)^D} e( - P_{A/Q}(r) ) \cdot \frac{1}{Q^D}  \sum_{B \in (Q)^D} e(-B/Q\cdot (r-n)) \cdot (m \chi_s)^{\vee}(n) \\
& \qquad \qquad \qquad = e(-P_{A/Q}(n)) \cdot (m \chi_s)^{\vee}(n),
\end{align*}
where we used the relationship
\begin{align*}
\frac{1}{Q^D} \sum_{B \in (Q)^D} e( B/Q \cdot x) = \begin{cases} 1 & \text{ if } x \equiv 0 \mod Q \\
0 & \text{ otherwise. }
\end{cases}
\end{align*}
\end{proof}

With this notation in mind, we may express
\begin{align*}
    L_\lambda^{s;j_0} = \mathscr{L}_{s,A/Q}[ \Phi_{\lambda-A/Q}^{s;j_0}]
\end{align*}
where
\begin{align*}
    \Phi_\lambda^{s;j_0} = \sum_{2^{s/A_0} \leq j \leq j_0} \Phi_{j,\lambda}^*
\end{align*}
and $A/Q$ is the unique element with $Q \sim 2^s$ so that
\begin{align*}
    |\lambda_\alpha - \frac{A_\alpha}{Q}| \leq 2^{-10s-10},
\end{align*}
or an arbitrary element of the complement otherwise (note that in this case $L_\lambda^s(\beta) = 0$).

\begin{lemma}\label{l:smallgs}
    There exists an absolute $c = c(d,D,p) > 0$ so that for every $s \geq 1$
    \begin{align*}
        \| \sup_{A/Q : Q \sim 2^s} |\mathscr{L}_{s,A/Q}[1]^{\vee}*f| \|_{\ell^p(\mathbb{Z}^D)} \lesssim 2^{-cs} \|f\|_{\ell^p(\mathbb{Z}^D)}.
    \end{align*}
\end{lemma}
\begin{proof}
By interpolation with Lemma \ref{l:easymaj}, it suffices to prove the estimate at $\ell^2$. We apply a $TT^*$ argument, and are left to consider the kernel,
\begin{align*}
    \mathcal{K}_0(x,n) &= \sum_{B \in (Q)^D, \ B' \in (Q')^D} S(A/Q,B/Q) \overline{S(A'/Q',B'/Q')} \\
    & \qquad \times \sum_m e(B/Q (x-m)) e(-B'/Q'(n-m)) (\chi_s')^{\vee}(x-m) \overline{(\chi_s')^{\vee}}(n-m) 
\end{align*}
where $A/Q = A(x)/Q(x)$ and $A'/Q' = A'(n)/Q'(n)$; our job is to show that
\begin{align}\label{e:kern0}
    \| \mathcal{K}_0 \|_{\text{Ker}(\mathbb{Z}^D)} \lesssim 2^{-c_0 s}.
\end{align}

We will do so by bounding
\begin{align*}
    |\mathcal{K}_0(x,n)| \lesssim 2^{-c_0s} \rho_s(v) + (\mathbf{1}_{E_{\text{Per}}(A/Q)} \cdot \rho_s)(v)
\end{align*}
where
\[ \rho_s(x) := \sum_n (\chi_s')^{\vee}(x-n) \cdot (\chi_s')^{\vee}(n) \]
has spatial scale $2^{2^{10 \rho s}}$, and $E_{\text{Per}}(A/Q)$ is the $Q$-periodic extension of some subset, $E(A/Q) \subset (Q)^D$ depending only on $A/Q$ which has density $2^{-c_0s}$.

In particular, we bound
\begin{align*}
\sup_x \sum_n |\mathcal{K}_0(x,n)| \lesssim 2^{-c_0 s}, \; \; \; \sup_n \sum_x |\mathcal{K}_0(x,n)| \lesssim 1,
\end{align*}
so \eqref{e:kern0} follows from Schur's test.

Turning to $\mathcal{K}_0(x,n)$, since $\chi_s'$ has such a small Fourier support, the sum vanishes unless $B/Q = B'/Q'$, as can be seen by applying Poisson summation. The only way that can happen is if $Q|Q'$ or vice versa; in either event we would find $Q = Q'$, since both have size $\sim 2^s$. This leads to the diagonalization
\begin{align*}
    \mathcal{K}_0(x,n) &= \sum_{B \in (Q)^D} S(A/Q,B/Q) \overline{S(A'/Q,B/Q)} \cdot  e(B/Q(x-n)) \cdot \rho_s(x-n) \\
    & \qquad = \frac{1}{Q^D} \sum_{r \in (Q)^D} e( - P_{A/Q}(x-n + r) + P_{A'/Q}(r)) \cdot \rho_s(x-n).
\end{align*}
We claim that there exists an absolute $c_0>0 $ so that
\begin{align}\label{e:rare}
   |\frac{1}{Q^D} \sum_{r \in (Q)^D} e( - P_{A/Q}(v + r) + P_{A'/Q}(r))| \leq 2^{-c_0s}  + \mathbf{1}_{E(A/Q)}(v)
\end{align}
where $E(A/Q) \subset (Q)^D$ has density $2^{-c_0s}$; consequently
\begin{align*}
|\mathcal{K}_0(x,n)|  \leq 2^{-c_0s} \rho_s(x-n) + (\mathbf{1}_{E_{\text{Per}}(A/Q)} \cdot \rho_s)(x-n).
\end{align*}
where $E_{\text{Per}}(A/Q)$ is the $Q$-periodic extension of $E(A/Q)$. But, since $P$ has no linear terms
\begin{align*}
    N_Q(P_{A/Q}) \geq 2^{s-1} 
\end{align*}
since for any $|\alpha| \geq 2$
\begin{align*}
    \min_{q \leq 2^{s-2}} Q^{|\alpha|} \cdot \| q \frac{A_\alpha}{Q} \|_{\mathbb{T}} \geq Q^{|\alpha| -1 } > 2^{s-2};
\end{align*}
similarly, $N_Q(P_{A'/Q}) \geq 2^{s-1}$. So, \eqref{e:rare} follows from Lemma \ref{c:sublevel1}. The details are as follows: with 
\[ P_j(v) = \sum_\alpha \frac{A_\alpha}{Q} \alpha_j v^{\alpha - e_j},\]
a polynomial with coefficient norm $\gtrsim_d 2^s$, it suffices to show that
\begin{align}\label{e:rare1}
|\{ v \in (Q)^D : N_Q( \sum_{j=1}^D P_j(v) \cdot r^{e_j}) \leq Q^{\kappa} \}| \lesssim Q^{D - \kappa_0};
\end{align}
but the left-hand side of \eqref{e:rare1} is contained in
\begin{align*}
\bigcap_{j=1}^D \{ v \in (Q)^D : \min_{q \leq Q^{\kappa}} \| q \cdot P_j(v) \|_{\mathbb{T}} \leq Q^{\kappa - 1} \},
\end{align*}
which has measure bounded by $Q^{D - \kappa_0}$ by Lemma \ref{c:sublevel1}.
\end{proof}



We next recall the following Lemma from \cite{BK12}.
\begin{lemma}\label{l:borrow}
For $j \geq 1$ let $\mathcal{K}_j$ denote a mean-zero $\mathcal{C}^1$ function supported on $\{ |x| \approx 2^j\}$, with
\begin{align}
    2^{jD} |\mathcal{K}_j(x)| +     2^{j(D+1)} |\nabla \mathcal{K}_j(x)| \leq C
\end{align}
uniformly in $j \geq 1$. Let $\mathcal{K}^{a,b} := \sum_{a \leq j < b} \mathcal{K}_j$. Then there exists $c = c(d,D,p) > 0$ so that
\begin{align*}
    \| \sup_{A/Q,J \geq 1} |\mathscr{L}_{s,A/Q}[\widehat{\mathcal{K}^{0,J}}]^{\vee}* f | \|_{\ell^p(\mathbb{Z}^D)} \lesssim 2^{-cs} \|f\|_{\ell^p(\mathbb{Z}^D)}.
    \end{align*}
\end{lemma}
\begin{proof}
See \cite[Lemma 4.4]{BK12}, noting that the argument is invariant under rearrangement of the polynomial phase in the appropriate Gauss sums,
\[ S(A/Q,B/Q) = \frac{1}{Q^D} \sum_{r \in [Q]^D} e(- \sum_{\alpha} \frac{A_\alpha}{Q} r^\alpha - B/Q \cdot r),\]
subject to the estimate from Lemma \ref{l:smallgs} and the factorization \eqref{e:factor}.
\end{proof}

We now prove Proposition \ref{p:mainprop}.

\subsection{The Proof of Proposition \ref{p:mainprop}}
For $s \geq 1$ fixed, let
\begin{align*}
    \mathcal{J}_{l,\mu} := \{ j : j^{A_0} \geq 2^s, \ \| P_\mu(2^j \cdot) \| \sim 2^l \},
\end{align*}
see \eqref{e:Ecn}. Note that $|\mathcal{J}_{l,\lambda}| = O_{d}(1)$, see \cite[Lemma 2.1]{LX} for details; the key point is that there are only $O(d^2 \log A) = O_{d}(1)$ many scales $j$ so that
\begin{align*}
    A^{-1} \leq \frac{ \sum_{|\alpha| = k} 2^{jk} |\lambda_\alpha| }{\sum_{|\alpha| = k'} 2^{jk'} |\lambda_\alpha|} \leq A
\end{align*}
for $2 \leq k \neq k' \leq d$. By sparsifying our scales into $O_{d}(1)$ many sub-families, we will assume that
\begin{align*}
    \sup_{l} |\mathcal{J}_{l,\lambda}| \leq 1,
    \end{align*}
which we will index 
\begin{align*}
    \mathcal{J}_{l,\lambda} = \{ j_l\}.
\end{align*}

We collect
\begin{align*}
    \mathbb{L}(j_0) := \{ l \leq - C s : j_l \leq j_0 \},
\end{align*}
where $C = C_{d,D,p}$ is a sufficiently large constant,

It suffices to bound

\begin{align*}
    \| \sup_{A/Q,\mu,j_0} |\sum_{l \in \mathbb{L}(j_0)} \mathscr{L}_{s,A/Q}[\Phi_{j_l,\mu}]^{\vee}*f| \|_{\ell^p(\mathbb{Z}^D)} \lesssim 2^{-c s} \|f\|_{\ell^p(\mathbb{Z}^D)},
\end{align*}
and
\begin{align*}
    \| \sup_{A/Q,\mu} |\mathscr{L}_{s,A/Q}[\Phi_{j_l,\mu}]^{\vee}*f| \|_{\ell^p(\mathbb{Z}^D)} \lesssim 2^{-c s} \cdot \min\{ 1, 2^{-cl} \} \cdot  \|f\|_{\ell^p(\mathbb{Z}^D)}.
\end{align*}


We begin with the low frequency case; by direct computation, for any $A/Q$, we can express
\begin{align*}
    &\sum_n f(x-n) \sum_{l \in \mathbb{L}(j_0)} \sum_{B \in [Q]^D} S(A/Q,B/Q) \cdot e(B/Q n) \cdot \int \chi_s^{\vee}(n-t) e(-P_\mu(t)) \psi_{j_l}(t) \ dt \\ 
& \qquad = \sum_n f(x-n) \sum_{B \in [Q]^D} S(A/Q,B/Q) \cdot e(B/Q n) \cdot \int \chi_s^{\vee}(n-t) \Big( \sum_{l \in \mathbb{L}(j_0)}  \psi_{j_l}(t) \Big) \ dt \\
    & \qquad \qquad + O\Big(\sum_{l \leq -Cs} Q^D \cdot \sum_n |f(x-n)| \int |\chi_s^{\vee}(n-t)| |\psi_{j_l}(t)| |P_\mu(t)| \ dt \Big) \\
    & \qquad = \mathscr{L}_{s,A/Q}[\widehat{{K}^{J_-,J_+(\mu)}}]^{\vee}*f(x) \\
    & \qquad \qquad + O \Big( Q^D \cdot \sum_{l \leq -Cs} 2^l \cdot \sum_n |f(x-n)| \cdot \int |\chi_s^{\vee}(n-t)| |\psi_{j_l}(t)| \ dt \Big) \\
    & \qquad = \mathscr{L}_{s,A/Q}[\widehat{{K}^{J_-,J_+(\mu)}}]^{\vee}*f(x) 
    + O\Big(2^{sD} \cdot 2^{-C s} \cdot M_{HL} f(x) \Big),
\end{align*}
where $J_-:= \min\{ j : j^{A_0} \geq 2^s \}$ and
\begin{align*}
    J_+(\mu) := \min \big\{ j_0, \max\{ j : \| P_\mu(2^j \cdot) \| \leq 2^{-Cs} \} \big\},
\end{align*}
see \eqref{e:Ecn}.

By Lemma \ref{l:borrow}, we bound
\begin{align*}
    \| \sup_{A/Q,\mu, j_0} |\sum_{l \in \mathbb{L}(j_0)} \mathscr{L}_{s,A/Q}[\Phi_{j_l,\mu}]^{\vee}*f| \|_{\ell^p(\mathbb{Z}^D)} \lesssim 2^{-c s} \|f\|_{\ell^p(\mathbb{Z}^D)},
\end{align*}
provided that $C$ is chosen sufficiently large.

We now prove that for each $l$, 
\begin{align*}
    \| \sup_{A/Q,\mu} |\mathscr{L}_{s,A/Q}[\Phi_{j_l,\mu}]^{\vee}*f| \|_{\ell^p(\mathbb{Z}^D)} \lesssim 2^{-c s} \cdot \|f\|_{\ell^p(\mathbb{Z}^D)}.
\end{align*}

First, we observe that 

\begin{align*}
    \sup_{A/Q,\mu} | \mathscr{L}_{s,A/Q}[\Phi_{j_l,\mu}]^{\vee}*f(x)| \lesssim M_{HL} f(x)
\end{align*}
by Lemma \ref{l:easymaj}. So, it suffices to exhibit the decay at the $\ell^2$ level.

We use the method of $TT^*$:
for an appropriate choice of linearizing functions,
\begin{align*}
A(x)/Q(x) : \mathbb{Z}^D \to \{ A/Q : Q \sim 2^s \}
\end{align*}
and
\begin{align*}
    P_{\mu(x)} : \mathbb{Z}^D \to \mathscr{P}_{d,D}
\end{align*}
we may bound
\begin{align*}
    \sup_{A/Q,\mu} |\sum_{l \in \mathbb{L}} \mathscr{L}_{s,A/Q}[\Phi_{j_l,\mu}]^{\vee}*f(x)| \lesssim |\sum_n K(x,n) f(n)|,
\end{align*}
where
\begin{align*}
    K(x,n) &= \sum_{B \in (Q)^D} S( A(x)/Q(x), B/Q) e(B/Q(x-n)) \\
    & \qquad \times \int \chi_s^{\vee}((x-n) -t) e(-P_{\mu(x)}(t)) \psi_{j_l}(t) \ dt
\end{align*}
We exhibit an absolute $c > 0$ so that  
\begin{align}\label{e:kernelest0}
 \| K  \|_{\text{Ker}(\mathbb{Z}^D)} \lesssim 2^{-c  s}  
\end{align}
by $TT^*$. In particular, we will show that the integral operator with kernel
\begin{align*}
    & \mathcal{K}(x,z) = \sum_n K(x,n) \overline{K(z,n)} \\
    & =
    \sum_{B \in (Q)^D, B' \in (Q')^D } S( A(x)/Q(x), B/Q) \overline{S(A'(z)/Q'(z),B'/Q')} e(B/Q x - B'/Q' z) \\
    & \qquad  \times  \int 
    \Big( \sum_n e(- (B/Q- B'/Q') n) \chi_s^{\vee}(x-n-t) \overline{\chi_s^\vee}(z-n-u) \Big) \\
    & \qquad \qquad \times e(-P_{\mu(x)}(t) + P_{\mu'(z)}(u)) \psi_{j_l}(t) \overline{\psi_{j_l}}(u) \ dt du
\end{align*}
satisfies
\begin{align}\label{e:schur1}
    \sup_x \sum_z |\mathcal{K}(x,z)| \lesssim 2^{-c_0 s}, \; \; \; \sup_z \sum_x|\mathcal{K}(x,z)| \lesssim 1,
\end{align}
at which point we can bound $\| \mathcal{K}  \|_{\text{Ker}(\mathbb{Z}^D)} \lesssim 2^{-c_0 /2 s}$ by Schur's test, from which \eqref{e:kernelest0} follows with $c = \frac{c_0}{4}$.

But, since $\chi_s$ has such small support, $\mathcal{K}(x,z)$ diagonalizes:
\begin{align*}
    \mathcal{K}(x,z) &= \sum_{B \in (Q)^D} S(A/Q,B/Q) \overline{S(A'/Q,B/Q)} e(B/Q(x-z)) \\
    & \qquad \times \int \rho_s(x-z) e(-P_{\mu(x)}(t) + P_{\mu'(z)}(u)) \psi_{j_l}(t) \overline{ \psi_{j_l}(u) } \ dt du \\
& \qquad \qquad + O(2^{-2^{10 \rho s}} \cdot 2^{-j_l D} \mathbf{1}_{|x-z| \lesssim 2^{j_l}})
\end{align*}
where the error term arises from approximating
\[ \sum_n \chi_s^{\vee}(a - n) \chi_s^{\vee}(b-n) \]
by
\begin{align*}
    \rho_s(a-b) = \sum_n \chi_s^{\vee}(a-b - n) \chi_s^{\vee}(n),
\end{align*}
using the smoothness of $\chi_s^{\vee}$ at spatial scales of the order $2^{2^{10 \rho s}}$.
By \eqref{e:rare}, we may bound
\begin{align*}
|\sum_{B \in (Q)^D} S(A/Q,B/Q) \overline{S(A'/Q,B/Q)} e(B/Q(x-z))| \lesssim 2^{-c_0 s} + \mathbf{1}_{E_{\text{Per}}(A/Q)}(x-z)
\end{align*}
where $E_{\text{Per}}(A/Q)$ is the $Q$-periodic extension of a subset $E(A/Q) \subset (Q)^D$ which depends only on $A/Q$, and has relative density $2^{-cs}$. Consequently
\begin{align*}
    |\mathcal{K}(x,z)| &\leq 2^{-c_0s} \cdot \int |\rho_s((x-z) - (t-u))| |\psi_{j_l}(t) \psi_{j_l}(u)| \ dtdy\\
& \qquad  + \int |\rho_s((x-t) - (z-u))| |\psi_{j_l}(t) \psi_{j_l}(u)| \ dtdy \cdot \mathbf{1}_{E_{\text{Per}}(A/Q)}(x-z).
\end{align*}
which establishes \eqref{e:schur1}, given the small Lipschitz norm of $\lesssim 2^{-2^{10 \rho s}}$ of
\[ v \mapsto \int |\rho_s(v-t+u)| |\psi_{j_l}(t) \psi_{j_l}(u)| \ dt du. \]

Finally, we just observe that for each $\theta \in \mathbb{T}^D$
\begin{align*}
    &\| \sup_\mu| \int \Phi_{l,\mu}(\beta-\theta) \chi_s(\beta - \theta) \hat{f}(\beta) e(\beta x) | \|_{\ell^p(\mathbb{Z}^D)} \\
& \qquad = \| \sup_\mu| \int \Phi_{l,\mu}(\beta) \chi_s(\beta) \hat{f}(\beta+\theta) e(\beta x) | \|_{\ell^p(\mathbb{Z}^D)}  \\
& \qquad \qquad  \lesssim 2^{-c l} \|f\|_{\ell^p(\mathbb{Z}^D)}, \; \; \; c = c_{d,D,p} > 0
\end{align*}
by applying Magyar-Stein-Wainger transference \cite{MSW} and the continuous result of Stein-Wainger \cite{SW}. Summing appropriately
\begin{align*}
    &\| \sup_{A/Q,\mu} | \mathscr{L}_{s,A/Q}[\Phi_{l,\mu}]^{\vee}*f | \|_{\ell^p(\mathbb{Z}^D)} \\
&\qquad  \leq \sum_{A/Q : Q \sim 2^s}     \| \sup_{\mu} | \mathscr{L}_{s,A/Q}[\Phi_{l,\mu}]^{\vee}*f | \|_{\ell^p(\mathbb{Z}^D)} \\
    & \qquad  \qquad \leq 2^{s(|\Gamma|+D)} \cdot \sup_\theta 
        \| \sup_\mu| \int \Phi_{l,\mu}(\beta-\theta) \chi_s(\beta - \theta) \hat{f}(\beta) e(\beta x) | \|_{\ell^p(\mathbb{Z}^D)} \\
        &  \qquad \qquad \qquad \lesssim 2^{s(|\Gamma| +D)} \cdot 2^{-cl} \cdot \| f\|_{\ell^p(\mathbb{Z}^D)}.
\end{align*}
In particular, for any $l \geq 1$, we may bound
\begin{align*}
\| \sup_{A/Q,\mu} |\mathscr{L}_{s,A/Q}[\Phi_{j_l,\mu}]^{\vee}*f| \|_{\ell^p(\mathbb{Z}^D)} &\lesssim \min \{ 2^{-c s}, 2^{s (|\Gamma|+ D)} \cdot 2^{-cl} \} \cdot \|f\|_{\ell^p(\mathbb{Z}^D)} \\
& \qquad \lesssim 2^{-c_0 (s+l)} \cdot \|f\|_{\ell^p(\mathbb{Z}^D)},
\end{align*}
after interpolating appropriately.

\appendix
\section{The Proof of Theorem \ref{p:coeffnorm}}\label{s:app}
In this appendix we provide a full proof of Theorem \ref{p:coeffnorm} by establishing the following inverse theorem, Theorem \ref{t:Inv}.

Multi-dimensional arithmetic progressions inside of $[ \vec{N} ]$ will be indexed as
\begin{align}\label{e:ArithP}
    \mathcal{P} = \mathcal{P}_1 \times \dots \times \mathcal{P}_D \subset [\vec{N}]
\end{align}
provided that $\mathcal{P}_i \subset [N_i]$ are arithmetic progressions. We will use 
\begin{align}\label{e:sigma1}
   \sigma_i := \min_{p \neq p' \in \mathcal{P}_i} \, |p - p'| 
\end{align}
to denote the gap sizes of $\{ \mathcal{P}_i\}$.

Then our result is as follows.

\begin{thm}\label{t:Inv}
Suppose that 
\begin{align*}
    |\frac{1}{|\vec{N}|} \sum_{n_i \in \mathcal{P}_i} e(P(n))| \geq \delta 
\end{align*}
for some arithmetic progressions $\mathcal{P}_i \subset [N_i]$ with gap sizes $\sigma_i \leq \delta^{-1}$, see \eqref{e:sigma1}.
Then either
\begin{itemize}
    \item For some $i$, $N_i = \delta^{-O(1)}$; or
    \item There exists some $Q \lesssim \delta^{-O(1)}$ so that for each $\lambda_\alpha$
    \begin{align*}
         \| Q \lambda_{\alpha} \|_{\mathbb{T}} \leq \frac{\delta^{-O(1)}}{\vec{N}^{\alpha}}. \end{align*}
\end{itemize}
\end{thm}

The proof we provide proceeds by a double induction on the degree of $P$, $d$, and the on the dimension of the ambient space, $D$, as well. The $D=1$ case of Theorem \ref{t:Inv} appears in \cite[\S 1]{T0}, and in any event follows the inductive arguments used below (the base case $d=D=1$ is again trivial); since the $d=1$ case holds for any $D$ by direct computation, we may assume that Theorem \ref{t:Inv} holds for all polynomials of degree $< d$ in every dimension, and that Theorem \ref{t:Inv} holds for all polynomials of degree $d$ in $<D$ dimensions.

The following general Hilbert-space lemma provides the main mechanism to induct downwards. We recall the Fej\'{e}r kernel at scale $K$:
\begin{align}\label{e:mu}
\mu_K(n) := \frac{1}{K} ( 1 - \frac{|n|}{K})_+.
\end{align}

\begin{lemma}[van der Corput's inequality, Special Case]\label{O-l:vdCcts}
The following estimate holds for any phase $P$, and any $0 \leq H \leq |I|$.
\[ |\frac{1}{|I|} \sum_{n \in I} e(P(n)) |^2 \lesssim \sum_k \mu_H(k)\cdot \Big| \frac{1}{|I|} \sum_{I \cap (I-k)} e(P(n+k) - P(n)) \Big| + (\frac{H}{|I|})^2.\]
\end{lemma}
\begin{proof}
Set $F(n) := e(P(n)) \cdot \mathbf{1}_{I}(n)$, and observe that
\begin{equation}\label{e:II}
\begin{split}
\mathcal{I} &:= \frac{1}{|I|} \sum_{n \in I} F(n) \\
& \qquad = \frac{1}{|I|} \sum_{n \in I} F(n+h) \ dt + O(\frac{H}{|I|}) 
\end{split}
\end{equation}
for any $h \leq H$. In particular,
\[ \mathcal{I} = \frac{1}{|I|} \sum_{n \in I} \Big( \frac{1}{H} \sum_{h \leq H} F(n+h) \Big) + O(\frac{H}{|I|}),\]
so by Cauchy-Schwartz
\[ |\frac{1}{|I|} \sum_{n \in I} F(n) |^2 \lesssim \frac{1}{|I|} \sum_{n \in I} \Big| \frac{1}{H} \sum_{h \leq H} F(n+h) \Big|^2 \ dt + O((\frac{H}{|I|})^2);\]
note how we used that the support constraint on $F$ implies that
\[ n \mapsto \frac{1}{H} \sum_{h \leq H} F(n+h) \]
is supported in $3I$. We expand the integral and change variables to conclude.
\end{proof}

In the discrete setting, passing to appropriate subsets of arithmetic progressions plays the role of rescaling. Since the mechanism of passing to a small sub-interval of an arithmetic progression with small gap size will be used often, we introduce the following definition.

\begin{definition}
Suppose that $[\vec{N}]$ is given, and let $\delta >0$ be a small number. Given two multi-dimensional arithmetic progressions, $\mathcal{P}', \mathcal{P} \subset [\vec{N}]$, we say that $\mathcal{P}'$ is
a \emph{$\delta$-rescaling} of $\mathcal{P}$ if $\mathcal{P}' \subset \mathcal{P}$, and 
\begin{align*}
    |\mathcal{P}_i'| \leq \delta \cdot |\mathcal{P}_i|, \; \; \; \sigma_i' \geq \delta^{-1} \sigma_i
\end{align*} 
for each $1 \leq i \leq D$, see \eqref{e:sigma1}. 
\end{definition}

The following lemma will be used often after rescaling in various Taylor expansion arguments along arithmetic progressions.

\begin{lemma}\label{l-Taylor}
Suppose that $Q$ is such that 
\[ \| \lambda_\alpha Q \|_{\mathbb{T}} \leq \Delta \cdot \vec{N}^{-\alpha}.\]
Suppose that $l_i \in [M_i]$ with $M_i \leq N_i$. Then
\begin{align*}
P(t_0 + lQ) = P(t_0) + O(\Delta \cdot \sum_{i=1}^D \frac{M_i}{N_i} \cdot Q^{d-1}).
\end{align*}
\end{lemma} 
\begin{proof}
Set $\mu := \sum_{i=1}^D \frac{M_i}{N_i}$. For each $\alpha$,
\begin{align*}
\lambda_\alpha (t_0 + lQ)^\alpha &= \lambda_\alpha t_0^\alpha + \lambda_\alpha Q \cdot \Big( \sum_{\beta < \alpha} \binom{\alpha}{\beta} t_0^{\beta} l^{\alpha - \beta} Q^{|\alpha - \beta| -1} \Big) \\
& \qquad = \lambda_\alpha t_0^\alpha + O\big( \Delta \vec{N}^{-\alpha} \cdot Q^{d-1} \vec{N}^{\alpha}  \cdot \sum_{0 < |\beta| \leq \alpha} O(\frac{\vec{M}^\beta}{\vec{N}^\beta}) \big) \\
& \qquad \qquad = \lambda_\alpha t_0^\alpha + O\big( \Delta \cdot Q^{d-1} \cdot \mu \big),
\end{align*}
so the result follows by summing.
\end{proof}

We will also require a ``condensation of singularities" lemma, which appears as \cite[Lemma 1.1.14]{T0}. The content is that if one begins with a frequency which lives relatively close to many cyclic subgroups of not-too-large height, then it must live extremely close to some cyclic subgroup with extremely small height. 

\begin{lemma}\label{O-l:conds}
Suppose that $0 < \epsilon \ll \delta \ll 1$, and that $N \gg \delta^{-1}$. Suppose that there exists a subset $H \subset [N]$ with $|H| \geq \delta N$ so that for all $n \in H$,
\[ \| n \alpha_0 \|_{\mathbb{T}} \leq \epsilon.\]
Then there exists some $q \leq \delta^{-1}$ so that
\[ \| q \alpha_0 \|_{\mathbb{T}} \lesssim \epsilon \cdot \frac{q}{\delta N}.\]
\end{lemma}


With these reductions in mind, we are prepared to prove Theorem \ref{t:Inv}.

\subsubsection{The Proof of Theorem \ref{t:Inv}}
We begin by reducing our attention to top order degrees.

\begin{lemma}\label{l:induct}
It suffices to establish Theorem \ref{t:Inv} only for coefficients
\begin{align*}
    \{ \lambda_\alpha : |\alpha| = d\}
\end{align*}
\end{lemma}
\begin{proof}
Let $P$ be an arbitrary degree $d$ polynomial, which we decompose as above as
\begin{align*}
P(n) = \sum_{j=1}^{d} P_j(n) =: \sum_{j=1}^d \big( \sum_{|\alpha| = j} \lambda_\alpha n^\alpha \big).
\end{align*}
We induct downwards on $|\alpha|$. Thus, let $d > j_0$ be arbitrary, and assume that Theorem \ref{t:Inv} holds for $|\alpha| > j_0$. Thus, we will assume that there exist $Q \leq \delta^{-C}$ so that
\begin{align*}
    \| Q \lambda_\alpha \|_{\mathbb{T}} \lesssim \frac{\delta^{-C}}{\vec{N}^{\alpha}}
\end{align*} for all $|\alpha| > j_0$. Set 
\begin{align*}
K := Q \cdot \prod_{i=1}^D \sigma_i = \delta^{-O(1)},
\end{align*}
and subdivide $\mathcal{P}$ into 
\begin{align*}
\mathcal{P} = \bigcup_{k \leq \delta^{-O(DA)}} \mathcal{Q}_k \cup I, \; \; \; |I| \leq \delta^{AD} |\vec{N}|
\end{align*}
where each $\mathcal{Q}_k$ is a $\delta^A$-rescaling of $[\vec{R}]$ all with common gap size $\sigma_{i} = K$.
By the pigeon-hole principle, there exists some $\mathcal{Q}$ with lengths $N_i' = \delta^A N_i$ so that
\begin{align}\label{e:lbdlevel}
\delta^2 \leq |\frac{1}{|\vec{N'}|} \sum_{n_i \in [N_i']} e( \sum_{k= j_0+1}^d P_k( r_{\mathcal{Q}} + K n) + \sum_{k=1}^{j_0} P_k( r_{\mathcal{Q}} + K n))|
\end{align}
for some $r_{\mathcal{Q}} \in [\vec{N}].$
By Lemma \ref{l-Taylor}, provided that $A = O_{d,D}(1)$ is sufficiently large
\begin{align*}
\sum_{k= j_0+1}^d P_k( r_{\mathcal{Q}} + K n) \equiv \sum_{k=j_0+1}^d P_k( r_{\mathcal{Q}} ) + O(\delta^{A/2}),
\end{align*}
so \eqref{e:lbdlevel} becomes
\begin{align*}
\delta^2 \lesssim |\frac{1}{|\vec{N'}|} \sum_{n_i \in [N_i']} e( \sum_{k=1}^{j_0} P_k( r_{\mathcal{Q}} + Kn))|.
\end{align*}
The polynomial
\begin{align*}
n \mapsto \sum_{k=1}^{j_0} P_k( r_{\mathcal{Q}} + Kn)) &= P_{j_0}( r_{\mathcal{Q}} + Kn) + \text{ Lower order terms in $n$} \\
& \qquad = \sum_{|\alpha| = j_0} \lambda_\alpha K^{j_0} n^\alpha + \text{ Lower order terms in $n$};
\end{align*}
by hypothesis, there exist $Q'  \leq \delta^{-O(1)}$ so that
\begin{align*}
\| Q' \lambda_\alpha K^{j} \| \lesssim \frac{\delta^{-O(1)}}{\vec{N'}^\alpha} = \frac{\delta^{-O(1)}}{\vec{N}^\alpha};
\end{align*}
setting $Q_0 = Q' K^{j} \lesssim \delta^{-O(1)}$ for $|\alpha| = j$ completes the proof.
\end{proof}

To close the induction, we decompose
\begin{align}\label{e:split0}
P(n) = P_{\neq D}(n) + \sum_{j=1}^{d} P_{j,D}(n),
\end{align}
where 
\begin{align}\label{ePjD}
P_{j,D}(n) := \sum_{|\alpha| = j : \alpha_D \neq 0} \lambda_\alpha n^\alpha
\end{align}
and $P_{\neq D}$ is defined by subtraction and is independent of the $D$th variable.

Below, with $\vec{N}$ fixed, we call a coefficient approximable, or $\delta$-approximable, if there exists an absolute $C$ so that
\begin{align*}
    \min_{q \leq \delta^{-C}} \| \lambda_\alpha q \|_{\mathbb{T}} \lesssim \frac{\delta^{-C}}{\vec{N}^\alpha}. 
\end{align*}

We will complete the proof of Theorem \ref{t:Inv} be completing the following program:

\begin{itemize}
\item {Base Case:} The coefficients of $P_{d,D}$ are approximable;
\item {Downwards Inductive Step:} The coefficients of each $P_{j,D}, \ 1 \leq j < d$ are approximable;
\item {Second Inductive Step:} The degree $d$ coefficients of $P_{\neq D}$ are approximable as well.
\end{itemize}

The second inductive step is the least invovled, so we dispose of it quickly.

\medskip

\begin{proof}[The Second Inductive Step:]
Assume that we have established the existence of $\{q_\alpha : \alpha_D \neq 0\}$ bounded above by $\delta^{-C}$, so that 
\begin{align*}
\| q_\alpha \lambda_\alpha \|_{\mathbb{T}} \lesssim \frac{\delta^{-O(1)}}{\vec{N}^{\alpha}}
\end{align*}
for each $\alpha : \alpha_D \neq 0$. Set
 \begin{align}\label{e:Q0}
Q_0 := \prod_{\alpha : \alpha_D \neq 0} q_\alpha \lesssim \delta^{-O(1)}
\end{align}
and, for $A$ sufficiently large, use the pigeon-hole principle to extract a $\delta^A$-rescaling of $[\vec{N}]$ with gap size $Q_0$, call it $\mathcal{P}$, so that
\begin{align}\label{e:proglbd}
\delta^C \lesssim \Big| \frac{1}{|\mathcal{P}|} \sum_{p:Q_0 p + r\in \mathcal{P}} e(P(Q_0 p + r)) \Big|.
\end{align}
By Lemma \ref{l-Taylor},
\begin{align*}
P(Q_0 p + r) &= P_{\neq D}(Q_0 p_1 + r_1, \dots, Q_0 p_D + r_D) + \sum_{j=1}^d P_{j,D}(Q_0 p_1 + r_1, \dots, Q_0 p_D + r_D) \\
& \qquad = P_{\neq D}(Q_0 p_1 + r_1, \dots, Q_0 p_D + r_D) + \sum_{j=1}^d P_{j,D}(r_1, \dots, r_D) + O(\delta^{A/2}),
\end{align*}
and thus the lower bound \eqref{e:proglbd} implies
\begin{align*}
\delta^C \lesssim \Big|\frac{1}{|\mathcal{P}|} \sum_{ p: Q_0 p + r\in \mathcal{P}} e(P_{\neq D}(Q_0 p + r)) \Big|  + O(\delta^{A/2})
\end{align*}
at which point the inductive hypothesis kicks in, as $p \mapsto P_{\neq D}(Q_0 p + r)$ is a degree $\leq d$ polynomial in at most $D-1$ many variables with leading order coefficients the same as
\[ n \mapsto Q_0^d \cdot P_{\neq D}(n),\]
and $Q_0 = \delta^{-O(1)}$.
\end{proof}

We now turn to the main argument.

\begin{proof}[The Base Case]
Our goal is to prove that the coefficients of  $P_{d,D}$ are approximable, see \eqref{e:split0}.

With $K = c_0 \delta \frac{N_D}{\sigma_D}$ for a sufficiently small constant $c_0$, we may express

\begin{align*}
    \frac{1}{|\vec{N}|} \sum_{n_i \in \mathcal{P}_i} e(P(n)) = \frac{1}{|\vec{N}|} \sum_{n_i \in \mathcal{P}_i} e(P(n+h \sigma_D \cdot e_D)) + O(c_0 \delta)
\end{align*}
uniformly in $h \in ( K )$, see \eqref{e:coord}. Averaging in $h \in (K)$ and applying Cauchy-Schwartz, we deduce a lower bound,
\begin{align*}
    \delta^2 \lesssim \sum_h \mu_{K}(h) \cdot \Big| \frac{1}{|\vec{N}|} \sum_{n_i \in \mathcal{P}_i, \ n_D \in (\mathcal{P}_D \cap \mathcal{P}_D - \sigma_D h)} e(P(n+h\sigma_D \cdot e_D) - P(n)) \Big|
\end{align*}    
see \eqref{e:proglbd}.

By the pigeon-hole principle, there exists some subset $H \subset [K] \subset [N_D/\sigma_D]$ of size 
\begin{align*}
    |H| \gtrsim K \approx \delta N_D/\sigma_D \gtrsim \delta^2 N_D
\end{align*}
so that for all $h \in H$
\begin{align*}
    \delta^2 \lesssim \Big| \frac{1}{|\vec{N}|} \sum_{n_i \in \mathcal{P}_i, \ n_D \in (\mathcal{P}_D \cap \mathcal{P}_i - \sigma_D h)} e(P(n+h\sigma_D \cdot e_D) - P(n)) \Big|.
\end{align*}
The polynomials
\begin{align*}\label{e:diff}
    n \mapsto P(n+h\sigma_D \cdot e_D) - P(n)
\end{align*}
are polynomials of degree $d-1$; by our inductive hypothesis, we know that for each $h \in H$ there exists some $q_{\alpha}(h) \lesssim \delta^{-C}$ so that 
\begin{align*}
    \| q_{\alpha}(h) \cdot (\lambda_\alpha \alpha_D \sigma_D) \cdot h \|_{\mathbb{T}} \lesssim \frac{\delta^{-C}}{\vec{N}^{\alpha - e_D}}
\end{align*}
for each $|\alpha| = d$, since the monomials
\[ n \mapsto \lambda_\alpha \alpha_D \sigma_D h \cdot n^{\alpha - e_D} \]
with $|\alpha| = d$ will appear as top order terms in \eqref{e:diff}.

By pigeon-holing appropriately, there exists some subset $H' \subset H$ of size $\gtrsim \delta^C N_D$ so that for each $h \in H'$, there exists a single $q_\alpha \lesssim \delta^{-C}$ so that 
\begin{align*}
    \| q_{\alpha} \cdot (\lambda_\alpha \alpha_D \sigma_D) \cdot h \|_{\mathbb{T}} \lesssim \frac{\delta^{-C}}{\vec{N}^{\alpha - e_D}}
\end{align*}
for each $|\alpha| = d_0$.

Set $Q = \prod_{\alpha : \alpha_D \neq 0} q_\alpha \cdot \alpha_D$, so that $Q \lesssim \delta^{-O_{d,D}(1)}$, and for each $|\alpha| = d$ so that $\alpha_D \neq 0$,
\begin{align*}
        \| Q \cdot (\lambda_\alpha \sigma_D) \cdot h \|_{\mathbb{T}} \lesssim \frac{\delta^{-C}}{\vec{N}^{\alpha - e_D}}
\end{align*}
We now apply Lemma \ref{O-l:conds}; specifically, with

\begin{align*}\epsilon := \frac{\delta^{-C}}{\vec{N}^{\alpha - e_D}},
\end{align*}
and 
\begin{align*}
    \alpha_0 = Q \cdot (\lambda_\alpha \sigma_D),
\end{align*}
we deduce the existence of an integer $q_0 \lesssim \delta^{-O_{d,D}(1)}$ so that 
\begin{align*}
    \| q_0 \alpha_0 \|_{\mathbb{T}} = \| (q_0 Q \sigma_D) \cdot \lambda_\alpha \|_{\mathbb{T}}  \lesssim \frac{\epsilon q_0}{\delta^C N^{e_D}} = \frac{\delta^{-O_{d,D}(1)}}{\vec{N}^\alpha}.
\end{align*}
Since $q_0 Q \sigma_D = \delta^{-O_{d,D}(1)}$, we have shown that for every $|\alpha| = d_0$ with $\alpha_D \neq 0$, there exists some $q = q_0 Q \sigma_D \lesssim \delta^{-O(1)}$ so that
\begin{align*}
\| q \lambda_\alpha \|_{\mathbb{T}} \leq \frac{\delta^{-O(1)}}{\vec{N}^{\alpha}}.
\end{align*}
\end{proof}

We now complete the proof by establishing our main inductive step.

\begin{proof}[The Downwards Inductive Step]
We here assume the existence of some $Q = \delta^{-O(1)}$ so that 
\begin{align*}
\| Q \lambda_\alpha \|_{\mathbb{T}} \lesssim \frac{\delta^{-O(1)}}{\vec{N}^{\alpha}}
\end{align*}
for all $|\alpha| > j_0$ with $\alpha_D \neq 0$, and our job is to extract some $\{q_\alpha \} \lesssim \delta^{-O(1)}$ so that 
\begin{align*}
\| q_\alpha \lambda_\alpha \|_{\mathbb{T}} \lesssim \frac{\delta^{-O(1)}}{\vec{N}^{\alpha}}
\end{align*}
for all $|\alpha| = j_0$ with $\alpha_D \neq 0$.

By the pigeon-hole principle, we can find a $\delta^A$-rescaling, $\mathcal{P'} \subset \mathcal{P}$, with $\sigma_i' = Q \sigma_i$,
so that 
\begin{align}\label{e:lbdsub}
\delta^C \lesssim | \frac{1}{|\mathcal{P}'|} \sum_{n_i \in \mathcal{P}_i'} e(P_{\neq D}(n) + \sum_{j=1}^{d} P_j(n) ) |.
\end{align}
Expressing
\begin{align*}
\mathcal{P}_i' \ni n_i = (\sigma_i Q) \cdot p + k_i, \; \; \; p \leq \delta^{A_0} N_i, \; \; \;  k_i \in [N_i]
\end{align*}
we compute that for each $j >j_0$
\begin{align*}
P_{j,D}(n) &= P_{j,D}(n_1,\dots,n_D) = P_{j,D}\big((\sigma_1 Q ) \cdot p_1 + k_1 ,\dots,(\sigma_D Q ) \cdot  p_D + k_D  \big) \\
& \qquad = P_{j,D}(k_1,\dots,k_D) + O(\delta^{A/2})
\end{align*}
provided that $A$ is sufficiently large, by another application of Lemma \ref{l-Taylor}. We therefore deduce
\begin{align}\label{e:lbdsub}
\delta^C \lesssim | \frac{1}{|\mathcal{Q}'|} \sum_{n_i \in \mathcal{Q}_i'} e(P_{\neq D}(n) + \sum_{j=1}^{j_0} P_j(n) ) |;
\end{align}
we now argue as above, differencing, pigeon-holing, and then applying Lemma \ref{O-l:conds} to exhibit a $Q_0 \lesssim \delta^{-O(1)}$ so that 
\begin{align*}
\| Q_0 \lambda_\alpha \|_{\mathbb{T}} \lesssim \frac{\delta^{-O(1)}}{\vec{N}^{\alpha}}
\end{align*}
for all $|\alpha| = j_0$ where $\alpha_D \neq 0$, closing the induction and completing the proof.
\end{proof}

\end{document}